\documentclass[12pt]{amsart}
\usepackage{amsmath,amsthm,amsfonts,amssymb,mathrsfs}%\usepackage{paralist,nicefrac}
\usepackage[all]{xy}                     % can't live without diagrams!!!
\UseTips                                        % Use Computer Modern Arrowheads in diagrams
\newdir^{ (}{{}*!/-3pt/\dir^{(}}
\newdir_{ (}{{}*!/-3pt/\dir_{(}}

\RequirePackage[dvipsnames,usenames]{color}  %makes the standard dvips color names available

\usepackage[total={6.5in,8.75in}, top=1.2in, left=1.0in, includefoot]{geometry}
\usepackage{amssymb}
\usepackage{latexsym}
\usepackage{amsmath,amsthm}
\usepackage{aeguill}
\usepackage{amssymb}
\usepackage{mathrsfs}
\usepackage{hyperref}
\usepackage{etoolbox}
\usepackage{enumerate}
\usepackage[dvipsnames]{xcolor}
\usepackage{tikz}[2013/12/13]

\makeatletter
\pretocmd{\chapter}{\addtocontents{toc}{\protect\addvspace{15\p@}}}{}{}
\pretocmd{\section}{\addtocontents{toc}{\protect\addvspace{3\p@}}}{}{}
\makeatother

\makeatletter
\def\@tocline#1#2#3#4#5#6#7{\relax
  \ifnum #1>\c@tocdepth % then omit
  \else
    \par \addpenalty\@secpenalty\addvspace{#2}%
    \begingroup \hyphenpenalty\@M
    \@ifempty{#4}{%
      \@tempdima\csname r@tocindent\number#1\endcsname\relax
    }{%
      \@tempdima#4\relax
    }%
    \parindent\z@ \leftskip#3\relax \advance\leftskip\@tempdima\relax
    \rightskip\@pnumwidth plus4em \parfillskip-\@pnumwidth
    #5\leavevmode\hskip-\@tempdima
      \ifcase #1
       \or\or \hskip .5em \or \hskip 1em \else \hskip 1.5em \fi%
      #6\nobreak\relax
    \dotfill\hbox to\@pnumwidth{\@tocpagenum{#7}}\par
    \nobreak
    \endgroup
  \fi}
\makeatother
\DeclareSymbolFont{bbold}{U}{bbold}{m}{n}
\DeclareSymbolFontAlphabet{\mathbbold}{bbold}

\newcommand{\C}{\mathbb{C}}
\newcommand{\N}{\mathbb{N}}
\newcommand{\Z}{\mathbb{Z}}

\newcommand{\Q}{\mathbb{Q}}
\newcommand{\F}{\mathbb{F}}
\newcommand{\A}{\mathbb{A}}

\newcommand{\Gal}{\operatorname{Gal}}

\newcommand{\Res}{\operatorname{Res}}
\newcommand{\Ind}{\operatorname{Ind}}

\newcommand{\ad}{\operatorname{ad}}
\newcommand{\ab}{\operatorname{ab}}

\renewcommand{\sc}{\operatorname{sc}}
\renewcommand{\ss}{\operatorname{ss}}
\newcommand{\der}{\operatorname{der}}

\newcommand{\Hom}{\operatorname{Hom}}

\newcommand{\End}{\operatorname{End}}

\newcommand{\conn}{\operatorname{conn}}

\newcommand{\GL}{\mathrm{GL}}
\newcommand{\PGL}{\mathrm{PGL}}

\def\CM{{\scriptscriptstyle{ \mathrm{CM}}}}
\def\cm{{ \mathrm{pab}}}
\def\ab{\mathrm{ab}}
\def\wab{\mathrm{wab}}
\def\Elambda{\overline{E}_\lambda}

\def\bM{\mathbf{M}}
\def\bG{\mathbf{G}}
\def\bH{\mathbf{H}}
\def\bS{\mathbf{S}}

\def\bT{\mathbf{T}}

\def\bpx{\begin{pmatrix}}
\def\epx{\end{pmatrix}}

\newcommand{\hto}{\hookrightarrow}

\DeclareMathOperator{\Frob}{Frob}

\newtheorem{thm}{Theorem}[section]

\newtheorem{cor}[thm]{Corollary}

\newtheorem{prop}[thm]{Proposition}

\newtheorem{remark}[thm]{Remark}
\newtheorem{lem}[thm]{Lemma}
\newtheorem{defi}[thm]{Definition}

\advance\textheight by .07cm
\begin{document}

\title[]{On coefficients, potentially abelian quotients,  and residual irreducibility of compatible systems}

\author{Gebhard B\"ockle}
\email{gebhard.boeckle@iwr.uni-heidelberg.de}
\address{
Heidelberg University\\
IWR, Im Neuenheimer Feld 205,
69120 Heidelberg, Germany}

\author{Chun-Yin Hui}
\email{chhui@maths.hku.hk, pslnfq@gmail.com}
\address{
Department of Mathematics\\
The University of Hong Kong\\
Pokfulam, Hong Kong}


\begin{abstract}
Let $\{\rho_\lambda:\Gal_K\to\GL_n(\Elambda)\}$ be a semisimple $E$-rational compatible system of a number field $K$.  
In a first step,
building upon the theory of %Chenevier's and Rouquier's work on 
pseudocharacters \cite{Ro96},\cite{Ch14}, %in a first step 
we attach to each $\rho_\lambda$ an algebraic monodromy group $\bG_\lambda$ defined over $E_\lambda$
and  also prove that the compatible system can be descended to a strongly $E'$-rational compatible system 
$\{\rho_{\lambda'}:\Gal_K\to\GL_n(E'_{\lambda'})\}$ 
for some finite extension $E'/E$.
Secondly, we demonstrate that the maximal %\textcolor[rgb]{0,0,1}
{potentially abelian} quotient 
 of $\bG_\lambda$ is independent of $\lambda$ in a strong sense.
Finally, as an application, we generalize
a  result of Patrikis--Snowden--Wiles on residual irreducibility of compatible systems.
\end{abstract}

\maketitle
%\tableofcontents

\section{Introduction}\label{s1}
\subsection{$E$-compatible systems and algebraic monodromy groups}\label{s1.1} 
Let $K$ and $E$ be number fields and let $\Sigma_K$ and $\Sigma_E$ be the corresponding sets of finite places.
Fix an algebraic closure $\overline K$ of $K$ and equip the absolute Galois group $\Gal_K:=\Gal(\overline K/K)$
with the Krull topology.
For $\lambda\in \Sigma_E$, write $E_\lambda$ as the $\lambda$-adic completion of $E$, 
$\ell$ as the residue characteristic of $\lambda$, and $S_\ell:=\{v\in\Sigma_K: v|\ell\}$.

A family of $n$-dimensional \emph{$\lambda$-adic representations} of $K$ indexed by a subset $\Pi_E\subset\Sigma_E$,
\begin{equation}\label{cs}
\{\rho_\lambda: \Gal_K\to \GL_n(\Elambda):~\lambda\in\Pi_E\},
\end{equation}
is called an \emph{$E$-rational compatible system} (in the sense of Serre) if
there exist a finite $S\subset\Sigma_K$ and a
polynomial $P_v(T)\in E[T]$ for each $v\in\Sigma_K\backslash S$ such that 
the following conditions hold.
\begin{enumerate}[(a)]
\item For each $\lambda$, the representation $\rho_\lambda$ is unramified outside $S\cup S_\ell$.
\item For each $\lambda$ and $v\in\Sigma_K\backslash (S\cup S_\ell)$, 
the characteristic polynomial of \emph{Frobenius at $v$} satisfies
$$\det(\rho_\lambda(\mathrm{Frob}_v)-T\cdot\text{Id})=P_v(T)\in  E[T].$$
\end{enumerate}
In this case, we also call \eqref{cs} an \emph{$E$-compatible system} or just a \emph{compatible system}.

Compatible systems of Galois representations 
can be attached to smooth projective varieties \cite{De74}
and they also arise from certain cuspidal automorphic representations of $\GL_n$, 
see \cite{BLGGT14},\cite{HLTT16},\cite{Sc15}. 
We say that the family \eqref{cs} is \emph{semisimple} (resp. \emph{abelian}, \emph{potentially abelian})
if $\rho_\lambda$ is semisimple (resp. abelian, potentially abelian) for all $\lambda\in\Pi_E$.
We call \eqref{cs} a {\em strongly $E$-rational compatible system} (or \emph{strongly $E$-compatible})
if it is $E$-compatible and up to some change of coordinates,
$\rho_\lambda(\Gal_K)\subset\GL_n(E_\lambda)$\footnote{In other words, $\rho_\lambda$ can be descended to an $E_\lambda$-representation. If $\rho_\lambda$ is also semisimple, such $E_\lambda$-representation is unique up to isomorphism by the Brauer-Nesbitt theorem.}
 for all $\lambda\in\Pi_E$.
This condition on \eqref{cs} is more restrictive than being $E$-compatible.
The Zariski closure of the image $\rho_\lambda(\Gal_K)$ in $\GL_{n,\Elambda}$,
denote by $\bG_\lambda$, is called the \emph{algebraic monodromy group}  of $\rho_\lambda$.
It is a subgroup of $\GL_{n,\Elambda}$.

\subsection{%\textcolor[rgb]{0,0,1}
{Potentially abelian} quotients and $\lambda$-independence}\label{s1.2}
If \eqref{cs} is semisimple and strongly $E$-compatible,
then the algebraic monodromy group $\bG_\lambda$ is naturally a reductive subgroup of $\GL_{n,E_\lambda}$.
One can ask, as in \cite{Se94} for motives, 
if there exists a \emph{common $E$-model} $\bG\subset\GL_{n,E}$ for 
the family $\{\bG_\lambda\subset\GL_{n,E_\lambda}\}_{\Pi_E}$ in the sense that
$\bG\times_E E_\lambda$ and $\bG_\lambda$ are conjugate in $\GL_{n,E_\lambda}$
for all $\lambda\in\Pi_E$. When $\rho_\lambda$ is abelian for some $\lambda$, 
such a $\lambda$-independence problem has an affirmative answer 
by Serre-Waldschmidt theory \cite{Se98},\cite{Wa81}.
In general, it is known that the component group $\pi_0(\bG_\lambda):=\bG_\lambda/\bG_\lambda^\circ$,
the rank and semisimple rank of $\bG_\lambda$ are independent of $\lambda$ \cite{Se81},\cite{Hu13}.
There have been a lot of studies on the $\lambda$-independence of the identity component $\bG_\lambda^\circ$ 
(and the construction of $\bG^\circ\subset\GL_{n,E}$)
under specific situations, e.g., see \cite{Mo17} for a survey on the Mumford-Tate conjecture for abelian varieties \cite{Mu66},
see \cite{LP90,LP92},\cite{Hu25} when the tautological representation of $\bG_\lambda^\circ$ is absolutely irreducible 
for all $\lambda$, and see \cite[$\mathsection1.2$]{HL25} for recent progress under some \emph{type A} assumption.

 Suppose \eqref{cs} is semisimple and only $E$-compatible. %(i.e., without the strongly $E$-rational condition). 
In subsection \ref{ssec:algmon}, we shall 
provide a canonical construction of \emph{the algebraic monodromy group 
$\bG_\lambda:=\bG_{\rho_\lambda,E_\lambda}$} defined over $E_\lambda$
together with a morphism $\Gal_K\stackrel{\rho_\lambda}{\rightarrow}\bG_\lambda(E_\lambda)$ (Definition \ref{def:Glambda-Erat})
with natural properties (Proposition \ref{prop:Glambda-andBC}).
Our construction builds on %Chenevier's  and Rouquier's 
the theory of pseudocharacters \cite{Ro96},\cite{Ch14}, 
%and \cite[Theorem 2.16]{Ch14}, 
whereas
the only conditions we need are the semisimplicity of $\rho_\lambda$ 
and that the characteristic polynomial 
of $\rho_\lambda(g)$ takes values in $E_\lambda[T]$ for all $g\in\Gal_K$. 
This construction is 
compatible with base change, 
but the group $\bG_\lambda$ is in general not a subgroup
of $\GL_{n,E_\lambda}$.
Define the following $E_\lambda$-quotients of $\bG_\lambda$ for all $\lambda\in\Pi_E$.

\begin{itemize}
\item $\bG_\lambda^{\cm}:=\bG_\lambda/[\bG_\lambda^\circ,\bG_\lambda^\circ]$, the maximal %\textcolor[rgb]{0,0,1}
{potentially abelian} quotient of $\bG_\lambda$.
\item $\bG_\lambda^{\ab}:=\bG_\lambda/[\bG_\lambda,\bG_\lambda]$, the maximal abelian quotient of $\bG_\lambda$.
\end{itemize}

We prove that the systems $\{\bG_\lambda^{\cm}\}_{\Pi_E}$ and $\{\bG_\lambda^{\ab}\}_{\Pi_E}$ are independent of $\lambda$ in a strong sense.

\begin{thm}\label{main1}
Let $K$ be a number field and $\{\rho_\lambda: \Gal_K\to \GL_n(\Elambda)\}_{\Pi_E}$ 
be a semisimple $E$-rational compatible system of $K$
with $\{\bG_\lambda\}_{\Pi_E}$ as the system of algebraic monodromy groups defined over $E_\lambda$
(Definition \ref{def:Glambda-Erat}).
\begin{enumerate}[(i)]
\item There exists a family of faithful $E_\lambda$-representations 
$\{f_\lambda: \bG_\lambda^{\cm}\to\GL_{k,E_\lambda}\}_{\Pi_E}$ 
such that 
$$\{\phi_\lambda^{\cm}:\Gal_K\stackrel{\rho_\lambda}{\rightarrow}\bG_\lambda(E_\lambda)\to\bG_\lambda^{\cm}(E_\lambda)\stackrel{f_\lambda}{\rightarrow}\GL_k(E_\lambda)\}_{\Pi_E}$$
is a semisimple potentially abelian (strongly) $E$-rational compatible system of $K$.
\item There exists a family of faithful $E_\lambda$-representations $\{r_\lambda: \bG_\lambda^{\ab}\to\GL_{m,E_\lambda}\}_{\Pi_E}$
such that 
$$\{\phi_\lambda^{\ab}:\Gal_K\stackrel{\rho_\lambda}{\rightarrow}\bG_\lambda(E_\lambda)\to\bG_\lambda^{\ab}(E_\lambda)\stackrel{r_\lambda}{\rightarrow}\GL_m(E_\lambda) \}_{\Pi_E}$$
is a semisimple abelian (strongly) $E$-rational compatible system of $K$.
\item There exist reductive groups $\bG^{\cm}$ and $\bG^{\ab}$ defined over $E$ 
such that
$$\bG^{\cm}\times E_{\lambda}\simeq \bG_\lambda^{\cm}\hspace{.3in}\text{and}\hspace{.3in}
\bG^{\ab}\times E_{\lambda}\simeq \bG_\lambda^{\ab}\hspace{.3in} \text{for all}~ \lambda\in\Pi_E.$$
\end{enumerate}
\end{thm}
\begin{remark}
Observe that the targets of $f_\lambda$ and $r_\lambda$ are $\GL_{k,E_\lambda}$ 
and not $\GL_{k,\Elambda}$. But note that we have some freedom in choosing $k$; see for instance Theorem~\ref{pabindep}.
\end{remark}

\subsection{Descent of coefficients and residual irreducibility}\label{s1.3}

%\subsection{A residual irreducibility result of Patrikis-Snowden-Wiles}\label{s1.3}

Suppose \eqref{cs} is semisimple and $E$-compatible.
For any finite field extension $E'/E$,  one obtains from \eqref{cs} a semisimple $E'$-compatible system
\begin{equation}\label{cscoeff}
\{\rho_{\lambda'}: \Gal_K\stackrel{}{\rightarrow} \GL_n(\overline E'_{\lambda'})\}_{\Sigma_{E'}(\Pi_E)}
\end{equation}
by a change of coefficients,
where $\Sigma_{E'}(\Pi_E):=\{\lambda'\in\Sigma_{E'}:~\lambda'~\text{divides}~\lambda\in\Pi_E\}$  (see subsection \ref{sExtcs}).
We say that \eqref{cs} can be \emph{descended to a strongly $E'$-rational compatible system}
if there exists $E'/E$ such that \eqref{cscoeff} is strongly $E'$-compatible.
This is true if \eqref{cs} is a
\emph{regular weakly compatible system}, see \cite[Lemma 5.3.1]{BLGGT14}.
%We prove that such descent holds in full generality.
We establish that such descent holds in complete generality, enabling its application to the semisimple 
$E$-compatible systems constructed in \cite{HLTT16},\cite{Sc15}.

\begin{thm}\label{coeffd}
Let $K$ be a number field and $\{\rho_\lambda: \Gal_K\to \GL_n(\overline E_\lambda)\}_{\Pi_E}$ 
be a semisimple $E$-rational compatible system of $K$. There exists a finite extension $E'/E$ such that the following assertions hold.
\begin{enumerate}[(i)]
\item The $E$-compatible system $\{\rho_\lambda\}_{\Pi_E}$ can be descended to 
a strongly $E'$-rational compatible system $\{\rho_{\lambda'}:\Gal_K\to \GL_n(E'_{\lambda'})\}_{\Sigma_{E'}(\Pi_E)}$.
\item The algebraic monodromy group $\bG_{\lambda'}$ of $\rho_{\lambda'}$ splits for all $\lambda'\in \Sigma_{E'}(\Pi_E)$.
\item For each $\lambda'$, every irreducible factor of $\rho_{\lambda'}$ is absolutely irreducible.
 \end{enumerate}
\end{thm}

For each $\lambda\in\Pi_E$, decompose $\rho_\lambda$  %over $\Elambda$ 
into a direct sum of irreducible representations:
\begin{equation}\label{decompose}
\rho_\lambda%\otimes \overline E_\lambda
=\bigoplus_i \rho_{\lambda,i}^{m_{\lambda,i}}
\end{equation}
so that $\rho_{\lambda,i}$ is irreducible for each $i$, and if $i\neq j$ 
then $\rho_{\lambda,i}$ and $\rho_{\lambda,j}$ are non-isomorphic.
Denote  by $\bar\rho_{\lambda,i}$ the \emph{residual representation} of $\rho_{\lambda,i}$\footnote{It is defined as 
the semisimplification of the reduction of $\rho_{\lambda,i}$ with respect to an integral lattice of the representation 
space. This is independent of the choice of the lattice by the Brauer-Nesbitt theorem.}.  One expects that the following statements hold.

\begin{enumerate}[(i)]
\item For all but finitely many $\lambda$, the residual representations $\bar\rho_{\lambda,i}$ remain irreducible. 
\item For all but finitely many $\lambda$, the residual representations $\bar\rho_{\lambda,i}$ and $\bar\rho_{\lambda,j}$
are non-isomorphic if $i\neq j$. 
\end{enumerate}
When \eqref{cs} is %a \emph{strictly compatible system} in the sense of \cite[$\mathsection5$]{BLGGT14}),
strongly $E$-compatible and satisfies certain local conditions,
statement (i) holds for those irreducible factors $\rho_{\lambda,i}$ 
with type A algebraic monodromy groups (e.g., when $\dim \rho_{\lambda,i}\leq 3$) \cite[Theorem 1.2]{Hu23}.
%For a subset $\Pi\subset \Sigma_\Q$, define $\Sigma_E(\Pi):=\{\lambda\in\Sigma_E: ~\lambda~\text{divides}~\ell\in\Pi\}$.
We present a density one result of the above expectation.

\begin{thm}\label{main2}
Let $K$ be a number field, $\Pi\subset\Sigma_\Q$ be a Dirichlet density one subset,
 and $\{\rho_\lambda: \Gal_K\to \GL_n(\overline E_\lambda)\}_{\Sigma_E(\Pi)}$ 
be a semisimple $E$-rational compatible system of $K$.
There exists a Dirichlet density one subset $\Pi'\subset\Pi$ such that 
the following assertions hold for the decomposition \eqref{decompose}. 
\begin{enumerate}[(i)]
\item For all $\lambda\in\Sigma_E(\Pi')$, the residual representations $\bar\rho_{\lambda,i}$ remain irreducible.
\item For all $\lambda\in\Sigma_E(\Pi')$, the residual representations $\bar\rho_{\lambda,i}$ and $\bar\rho_{\lambda,j}$
are non-isomorphic if $i\neq j$. 
\end{enumerate}
\end{thm}
Theorem~\ref{main2} is obtained by Patrikis-Snowden-Wiles \cite[Theorem 1.2]{PSW18} if the compatible system 
is strongly $E$-compatible and
satisfies the following Hodge-Tate condition:

\begin{enumerate}
\item[(HT)] There is a finite multi-set $I$ of integers such that 
whenever $v\notin S$ is a place of $K$ with residue characteristic equal to that of $\lambda\in\Sigma(\Pi)$,
then $\rho_\lambda$ is crystalline at $v$ with Hodge-Tate weights in the multi-set $I$.
\end{enumerate}

By using Theorem \ref{main1} and Theorem \ref{coeffd}(i), 
we deduce that the two conditions can be omitted in their proof.

\begin{remark}
In \cite[Lemma 7.1.3]{ACC+23}, it is stated that ``if $\mathcal R=\{r_\lambda\}$ is an extremely weakly compatible system of rank $2$
and $\mathcal R$ is irreducible, then for all $\ell$ in a set of Dirichlet density $1$ and all $\lambda|\ell$,
the residual representation $\bar{r}_\lambda$ is absolutely irreducible.''
This result is recovered as a special case of our Theorem \ref{main2}(i) for $n=2$.
\end{remark}

\subsection{Organization of the article}\label{s1.4}
In section 2, we describe several  constructions related to change of coefficients 
for algebraic morphisms and Galois representations, which are fundamental to the rest of the article.
In section 3, we utilize Chenevier's work \cite{Ch14} to construct the algebraic monodromy groups 
$\bG_{\rho_\lambda,E_\lambda}$ defined over $E_\lambda$ 
for a semisimple $E$-rational compatible system $\{\rho_\lambda\}$,
 and we establish some natural properties of these groups. 
Building on this construction, we
prove Theorem \ref{coeffd} by employing \cite{LP92},\cite{Yu13},\cite{Cl37}.
In section 4, we present the main results of the 
abelian theory of $\ell$-adic representations \cite{Se98}
and the recent developments on \emph{weak abelian direct summands} \cite{BH25}; the latter serves as a motivation of this article.
Next in section 5, we investigate the relationship between 
semisimple %\textcolor[rgb]{0,0,1}
{potentially abelian} $\ell$-adic representations and $\ell$-adic 
realizations of some \emph{CM motives with coefficients}, using \cite{DMOS82},\cite{Sc88},
and we prove a $\lambda$-independence result on the algebraic monodromy groups (Theorem \ref{pabindep}).
By integrating the results from section 2 through 5, we establish Theorem \ref{main1} in section 6. 
Finally in section 7, we deduce Theorem \ref{main2} following the strategy of \cite{PSW18},
 in conjunction with Theorems \ref{coeffd} and \ref{main1}.

\section{Basic definitions and constructions}\label{sBasic}
We present some basic definitions and constructions of representations of $\Gal_K$ and algebraic groups.
Let $K$ and $E$ be number fields. 
The notation in subsection \ref{s1.1} remains in force.

\subsection{$E$-rationality}\label{sErat}
An $\ell$-adic representation of $K$ is a continuous group homomorphism
\begin{equation}\label{rho}
\rho_\ell:\Gal_K\to\GL_n(\overline \Q_\ell).
\end{equation}
An embedding $E\to\overline\Q_\ell$ (e.g., when $E$ is a subfield of $\overline\Q_\ell$) defines a place $\lambda\in\Sigma_E$ above $\ell$,
inducing an identification $\Elambda \simeq \overline\Q_\ell$. 
If we fix a choice of embedding, 
we obtain a $\lambda$-adic representation
$$\rho_\lambda:\Gal_K\stackrel{\rho_\ell}{\rightarrow}\GL_n(\overline \Q_\ell)\stackrel{\simeq}{\rightarrow}\GL_n(\Elambda)$$
so that the two representations $\rho_\ell$ and $\rho_\lambda$ are identified.
In this case, $\rho_\ell$ is said to be 
\emph{$E$-rational} if there is a finite $S\subset\Sigma_K$
such that $\rho_\ell$ is unramified outside $S$ and 
\begin{equation}\label{char}
\det(\rho_\lambda(\mathrm{Frob}_v)-T\cdot\text{Id})\in E[T]
\end{equation}
for any $v\in\Sigma_K\backslash S$; $\rho_\ell$ is said to be \emph{strongly $E$-rational} if it is $E$-rational and 
up to some change of coordinates,
$\rho_\lambda(\Gal_K)\subset \GL_n(E_\lambda)$.

One sees that the $E$-rational compatible system \eqref{cs}
is naturally a family of $E$-rational representations such that the polynomials \eqref{char} satisfy a certain compatibility condition.
The category of $E$-rational $\ell$-adic representations of $K$ (resp. $E$-rational compatible systems of $K$ indexed by $\Pi_E$)
is closed under taking (finite) direct sum, tensor product, symmetric powers, alternating powers, and dual. 
 $E$-rationality is also preserved under restriction to $\Gal_L\subset\Gal_K$ for $L\supset K$ a finite extension. The last two assertions also hold with strongly $E$-rational in place of $E$-rational. Moreover, we have the following uniqueness result.

\begin{thm}\label{unique}(see \cite[Chap. I, $\mathsection2.3$, Theorem]{Se98})
Let $\{\rho_\lambda:\Gal_K\to\GL_n(\Elambda)\}_{\Sigma_E}$ and $\{\rho'_\lambda:\Gal_K\to\GL_n(\Elambda)\}_{\Sigma_E}$
be two semisimple $E$-rational compatible systems. If $\rho_{\lambda_0}\simeq \rho'_{\lambda_0}$
for some $\lambda_0\in\Sigma_E$, then $\rho_{\lambda}\simeq \rho'_{\lambda}$ for all $\lambda\in\Sigma_E$. 
The assertion also holds for semisimple strongly $E$-rational systems 
with isomorphism $\rho_{\lambda_0}\simeq \rho'_{\lambda_0}$  over $E_{\lambda_0}$.
\end{thm}

\subsection{Weil restriction of scalars}\label{sExt}
Let $F'/F$ be a finite extension of fields. 
By \emph{Weil restriction of scalars}, there exists a linear algebraic group $\mathrm{Res}_{F'/F}\GL_{k,F'}$ 
defined over $F$ such that 
$$\mathrm{Res}_{F'/F}\GL_{k,F'}(A)=\GL_k(A\otimes F')$$
for any commutative $F$-algebra $A$. Weil restriction is right adjoint to scalar extension \cite[Chap. V, Proposition 5.1]{Mi12}, i.e., for $F$-group schemes $\bG$ there is a natural bijection %canonical correspondence
\begin{equation}\label{adjgp}
 \iota:\mathrm{Hom}_{F'}(\bG\times F',\GL_{k,F'})  \stackrel{\simeq}{\rightarrow} \mathrm{Hom}_F(\bG, \mathrm{Res}_{F'/F}\GL_{k,F'}).
\end{equation}
%for any affine $F$-group scheme $\bG$.% \cite[Chap. V, Proposition 5.1]{Mi12}.

For instance by choosing an $F$-basis of $F'$, one may functorially identify $\GL_k(A\otimes F')$ with a subgroup of $\GL_{k[F':F]}(A)$, and thereby $\mathrm{Res}_{F'/F}\GL_{k,F'}$ becomes a subgroup of $\GL_{k[F':F],F}$.
%By identifying $\mathrm{Res}_{F'/F}\GL_{k,F'}$ as a subgroup of $\GL_{k[F':F],F}$,
This yields for any $F'$-representation $\phi':\bG\times F'\to \GL_{k,F'}$ in the left-hand side of \eqref{adjgp}, 
an $F$-representation of $\bG$ by %as follows:
\begin{equation}\label{eq:RS}
\mathrm{Res}_{F'/F}(\phi'):=[\bG\stackrel{\iota(\phi')}{\longrightarrow} \mathrm{Res}_{F'/F}\GL_{k,F'}\subset \GL_{k[F':F],F}].
\end{equation}

Here are two useful facts.

\begin{prop}\label{see'}
The $F$-representation $\mathrm{Res}_{F'/F}(\phi'): \bG(F) \to \GL_{k[F':F]}(F)$ of the group $\bG(F)$
is equivalent to the $F$-representation on $(F')^k$ via $\phi'$.
\end{prop}

When $F'/F =E'/E$  is a finite extension of number fields,
one has the following.

\begin{prop}\label{see}
For $\lambda\in\Sigma_E$, the representations $\mathrm{Res}_{E'/E}(\phi'): \bG(E_\lambda) \to \GL_{k[E':E]}(E_\lambda)$ and 
$$\bG(E_\lambda)\stackrel{\Delta}{\longrightarrow}\prod_{\lambda'|\lambda}\bG(E'_{\lambda'})
\stackrel{\prod_{\lambda'|\lambda}(\phi'\otimes E'_{\lambda'})}{\longrightarrow}\prod_{\lambda'|\lambda}\GL_k(E'_{\lambda'})
=\mathrm{Res}_{E'/E}\GL_{k,E'}(E_\lambda)\subset \GL_{k[E':E]}(E_\lambda)$$
are equivalent, where $\Delta$ denotes the diagonal map.
\end{prop}

\subsection{Extension and restriction of scalars of compatible systems}\label{sExtcs}
Let $E'/E$ be a finite extension.
For a subset $\Pi_E\subset\Sigma_E$, we define $\Sigma_{E'}(\Pi_E):=\{\lambda'\in\Sigma_{E'}:~\lambda'~\text{divides}~\lambda\in\Pi_E\}$.

Given a strongly $E$-rational compatible system $\{\rho_\lambda: \Gal_K\stackrel{}{\rightarrow} \GL_n(E_\lambda)\}_{\Pi_E}$, the family
\begin{equation}\label{cs'}
\{\rho_{\lambda'}:=\rho_\lambda\otimes E'_{\lambda'}: \Gal_K\stackrel{\rho_\lambda}{\rightarrow} \GL_n(E_\lambda)
\hto\GL_n(E'_{\lambda'})\}_{\Sigma_{E'}(\Pi_E)}
\end{equation}
is a  strongly $E'$-rational compatible system of $K$ indexed by $\Sigma_{E'}(\Pi_E)$.
This construction is called the \emph{extension of scalars} of $\{\rho_\lambda\}$ with respect to $E'/E$.
Conversely, given a  strongly $E'$-rational compatible system 
$\{\rho_{\lambda'}: \Gal_K\stackrel{}{\rightarrow} \GL_n(E'_{\lambda'})\}_{\Sigma_{E'}(\Pi_E)}$ 
of $K$, the family
\begin{equation}\label{Rescs}
\{\rho_{\lambda}:=\prod_{\lambda'|\lambda}\rho_{\lambda'}: \Gal_K\stackrel{(\rho_{\lambda'})}{\longrightarrow}
\prod_{\lambda'|\lambda}\GL_n(E'_{\lambda'})=\mathrm{Res}_{E'/E}\GL_{n,E'}(E_\lambda)\subset \GL_{n[E':E]}(E_\lambda)\}_{\Pi_E}
\end{equation}
is a  strongly $E$-rational compatible system of $K$ indexed by $\Pi_E$. 
This construction is called the \emph{restriction of scalars} of $\{\rho_{\lambda'}\}$ with respect to $E'/E$.
Note that the property of being semisimple (resp. abelian, potentially abelian) of a compatible system 
is preserved under both constructions.

For each $\lambda'\in \Sigma_{E'}(\Pi_E)$, pick an field isomorphism 
$i_{\lambda'}:\overline E'_{\lambda'}\to \overline E_\lambda$
that is identity on $E_\lambda$. One can also extend a semisimple $E$-compatible system 
$\{\rho_\lambda: \Gal_K\stackrel{}{\rightarrow} \GL_n(\overline E_\lambda)\}_{\Pi_E}$
to a semisimple $E'$-compatible system by a change of coefficients:
\begin{equation}\label{cs''}
\{\rho_{\lambda'}: \Gal_K\stackrel{\rho_\lambda}{\rightarrow} \GL_n(\overline E_\lambda)
\stackrel{i_{\lambda'}^{-1}}{\rightarrow}\GL_n(\overline E'_{\lambda'})\}_{\Sigma_{E'}(\Pi_E)}.
\end{equation}
By the Brauer-Nesbitt theorem and semisimplicity, the isomorphism class of $\rho_{\lambda'}$ is independent of the choice of $i_{\lambda'}$.
If the system is furthermore strongly $E$-compatible, the constructions in \eqref{cs'} and \eqref{cs''} coincide
over the algebraic closures $\overline E'_{\lambda'}$.

\section{Coefficients of algebraic monodromy groups and compatible systems}\label{coeff}
Based on Chenevier \cite{Ch14}, we define algebraic monodromy groups $\bG_\lambda$ defined over $E_\lambda$
and prove  descent of coefficients (Theorem \ref{coeffd}) for semisimple $E$-compatible systems.

\subsection{Algebraic monodromy groups under $E$-rationality}\label{ssec:algmon}
Let $\rho_\ell:\Gal_K\to\GL_n(\overline \Q_\ell) $ be a semisimple $\ell$-adic representation of $K$, 
and let $F\subset\overline\Q_\ell$ be
a subfield containing $\Q_\ell$ such that 
\begin{equation}\label{eq:char}
\det(\rho_\ell(g)-T\cdot\text{Id})\in F[T]\quad\hbox{for all }g\in\Gal_K.
\end{equation}
The semisimple representation $\rho_\ell$ defines naturally a $\overline\Q_\ell$-valued \emph{determinant law} 
$$D_{\rho_\ell}:\overline\Q_\ell[\Gal_K]\stackrel{\rho_\ell}{\rightarrow}M_n(\overline\Q_\ell)\stackrel{\det}{\rightarrow}\overline\Q_\ell$$ 
in the sense of \cite[$\mathsection1.5$]{Ch14}.
By \eqref{eq:char} and \cite[Corollary 1.14]{Ch14}, the determinant law $D_{\rho_\ell}$ 
 factors through an (unique) $F$-valued determinant law:
\begin{equation}\label{Fdetlaw}
D_{\rho_\ell/F}:F[\Gal_K]\to F.
\end{equation} 
By \cite[Lemma 1.19(ii)]{Ch14}, the kernel of $D_{\rho_\ell/F}$, denoted $\mathrm{Ker}(D_{\rho_\ell/F})$, is the biggest two-sided ideal of $F[\Gal_K]$
such that $D_{\rho_\ell/F}$ factors through the quotient
\begin{equation}\label{quotS}
\pi_F:F[\Gal_K]\to F[\Gal_K]/\mathrm{Ker}(D_{\rho_\ell/F})=:S_F.
\end{equation}

The structural result \cite[Theorem 2.16]{Ch14} gives a description of the $F$-algebra $S_F$
and the factorization of \eqref{Fdetlaw} via~$S_F$.
In the present context, these results can be interpreted as follows.
\begin{enumerate}[(a)]
\item The algebra $S_F$ is isomorphic to a finite product $\prod_i S_i$ where each $S_i$ is a central simple algebra over a finite extension $F_i$ of $F$; we let $d_i$ be the degree of $S_i$ over $F_i$, so that $d_i^2=\dim_{F_i} S_i$, and we set $e_i:=[F_i:F]$.
\item For each $F$-embedding $\sigma:F_i\to\overline\Q_\ell$, choose an isomorphism $\alpha_\sigma:S_i\otimes_{F_i}\overline\Q_\ell\simeq M_{d_i}(\overline\Q_\ell)$, that induces an embedding $\iota_\sigma:S_i\to M_{d_i}(\overline\Q_\ell)$, and let $\Delta_{i,m}:S_i\to S_i^{m}$ be the diagonal embedding.
Then for each $i$ there is 
an integer $m_i>0$ such that 
\begin{enumerate}[(i)]
\item $\sum_i m_ie_id_i=n$, and
\item the $F$-linearization $F[\Gal_K]\to M_n(\overline\Q_\ell)$ 
of $\rho_\ell$ (i.e., $\sum a_i g_i\mapsto \sum a_i \rho_\ell(g_i)$)
is equivalent to the composition 
\begin{equation}\label{linearcompos}
\xymatrix{F[\Gal_K]\stackrel{\pi_F}{\rightarrow}S_F=\prod_i S_i \ar[r]^-{\oplus_i \Delta_{i,m_i}}  & \prod_i S_i^{m_i} \ar[rrr]^-{\oplus_i \oplus_{\sigma\in \Hom_F(F_i,\overline\Q_\ell)}  \iota_\sigma^{m_i}}&&& \prod_i (M_{d_i}(\overline\Q_\ell))^{m_ie_i} \hto M_n(\overline\Q_\ell).
%\ar@{^{ (}->}[r]^-{\mathrm{incl.}}&M_n(\overline\Q_\ell).
}
\end{equation}
\end{enumerate}
\end{enumerate}
%\begin{itemize}
%\item $S$ is isomorphic to a finite product $\prod_i S_i$ where each $S_i$ is a central simple algebra over a finite extension $F_i$ of $F$, and let $d_i$ be the degree of $S_i$ over $F_i$, so that $d_i^2=\dim_{F_i} S_i$, and $e_i:=[F_i:F]$.
%\item For each $F$-embedding $\sigma:F_i\to\overline\Q_\ell$, choose an isomorphism $\alpha_\sigma:S_i\otimes_{F_i}\overline\Q_\ell\cong M_{d_i}(\overline\Q_\ell)$, that induces an embedding $\iota_\sigma:S_i\to M_{d_i}(\overline\Q_\ell)$, and let $\Delta_{i,m}:S_i\to S_i^{m}$ be the diagonal embedding.
%Then for each $i$ there is an $F$-algebra homomorphism $r_i:F[\Gal_K]\to S_i$ and an integer $m_i>0$ such that (i) $\sum_i m_ie_id_i=n$ and (ii) the composition 
%\[\xymatrix{ F[\Gal_K]\ar[r]^-{\oplus_i r_i}   &\prod_i S_i \ar[r]^-{\oplus_i \Delta_{i,m_i}}  & \prod_i S_i^{m_i} \ar[rrr]^-{\oplus_i \oplus_{\sigma\in \Hom_F(F_i,\overline\Q_\ell)}  \iota_\sigma^{m_i}}&&& \prod_i (M_{d_i}(\overline\Q_\ell))^{m_ie_i} \ar@{^{ (}->}[rr]^-{\mathrm{inclusion}}&&M_n(\overline\Q_\ell) }\]
%is the linearization $F[\Gal_K]\to M_n(\overline\Q_\ell), \sum a_i g_i\mapsto \sum a_i \rho_\ell(g_i)$ of $\rho_\ell$.
%\end{itemize} 
Note that the assertion (b-ii) follows from the semisimplicity of both representations,
the fact that they 
have the same characteristic polynomials by \cite[Theorem 2.16]{Ch14}, and Brauer-Nesbitt.
%the semisimplicity of $\rho_\ell$: the representation $\rho_\ell$ and the composition \eqref{linearcompos} , 
%and so (b-ii) follows from semisimplicity and Brauer-Nesbitt. 
Note also that 
$S_i\otimes_F\overline\Q_\ell\simeq S_i\otimes_{F_i} F_i \otimes_F\overline\Q_\ell
\simeq \oplus_{\sigma\in\Hom_F(F_i,\overline\Q_\ell)} S_i\otimes_{F_i}^{\sigma}\overline\Q_\ell.$

Denote by $\underline{S_F}^\times$ (resp. $\underline{S_i}^\times$) the $F$-group scheme given on $F$-algebras 
$A$ by $A\mapsto (S_F\otimes_FA)^\times$ (resp. $A\mapsto (S_i\otimes_FA)^\times$).
By (b-ii), $\rho_\ell$ is equivalent to the composition
\begin{equation}\label{longcompos}
\xymatrix@C+.2pc{\Gal_K \stackrel{\pi_F}{\rightarrow}\underline{S}_F^\times(F)=  \prod_i \underline{S}_i^\times(F)  \, \ar[r]^-{\prod_i \Delta_{i,m_i}}  &\, \prod_i (\underline{S}_i^\times)^{m_i}(F) \;\ar[r]^-{\mathrm{via\,}F\to\overline\Q_\ell} &\;\prod_i (\underline{S}_i^\times)^{m_i}(\overline\Q_\ell) \hto \GL_n(\overline\Q_\ell).\!
%\ar@{^{ (}->}[r]^-{\mathrm{incl.}}&\GL_n(\overline\Q_\ell)].
}
\end{equation}

\begin{defi}\label{def:algmon}
Let $\rho_\ell:\Gal_K\to\GL_n(\overline\Q_\ell)$ be a semisimple $\ell$-adic representation satisfying \eqref{eq:char}.
The Zariski closure of the image %$({\prod_i \Delta_{i,m_i}} \circ {\prod_i r_i})  (\Gal_K) $
$\pi_F(\Gal_K)$ in the $F$-group scheme $\underline{S}_F^\times$ in \eqref{longcompos},
denoted by $\bG_{\rho_\ell,F}$, is called 
the algebraic monodromy group of $\rho_\ell$ defined over $F$.
\end{defi}

 \begin{remark}
As suggested by G.~Chenevier,
 the above construction can equally well be realized using the earlier work \cite{Ro96} of Rouquier
since we are in characteristic zero.
We follow \cite{Ch14} because we were more familiar with Chenevier's work.
%found the concise treatment in   \cite[Theorem 2.16]{Ch14} most helpful.}
\end{remark}

\begin{remark}
As suggested by A.~Shavali, $\bG_{\rho_\ell,F}$ can also be defined not using \cite{Ch14}.
More precisely, let $F'\supset F$ be a finite extension
such that $\rho_\ell$ is isomorphic to a representation \[\rho_\ell^{F'}:\Gal_K\to \GL_n(F')=(\Res_F^{F'}\GL_n)(F).\]
Then $\bG_{\rho_\ell,F}$ is isomorphic to the Zariski closure of $\rho_\ell^{F'}(\Gal_K)$ in $ \Res_F^{F'}\GL_n$ of 
$\rho_\ell^{F'}$.

\emph{This statement can be shown in two steps. Let the notation be as in (a) and (b) above, and suppose first that $F'$ contains all $F_i$ and splits all $S_i$ over $F_i$. 
Then $\prod_i (\underline{S}_i^\times)^{m_i}\times_{F}F'$ embeds into $\GL_n$, and by the universal property of the Weil restriction it
follows that $\rho_\ell^{F'}$ from above may be factored via $\prod_i (\underline{S}_i^\times)^{m_i}\to \Res_F^{F'}\GL_n$, and so in this case
the assertion is clear. Suppose now that $F'$ is arbitrary. Then one chooses $F''$ containing $F'$ and all $F_i$ and splitting all $S_i$.
In this case, using the first step, the claimed isomorphism will follow by comparing the construction via Weil restriction for $F'$ with that for $F''$. And 
for this comparison uses the canonical map $ \Res_F^{F'}\GL_n\to  \Res_F^{F''}\GL_n$ similarly to the first case.}

However the approach built on \cite{Ch14} is used in our proof below of Theorem~\ref{coeffd}, it simplifies the arguments 
for Proposition~\ref{prop:Glambda-andBC}, and it gives a conceptual definition of the semisimple algebra $S_F=\prod_i S_i$
that naturally arises from the linearization of $\rho_\ell$ over $F$ (via $D_{\rho_\ell/F}$).
\end{remark}
We present some natural properties of the (continuous) morphism $\Gal_K  \stackrel{\pi_F}{\rightarrow} \bG_{\rho_\ell,F}(F)$.

\begin{prop}\label{prop:Glambda-andBC}
Let $\rho_\ell:\Gal_K\to\GL_n(\overline\Q_\ell)$ be a semisimple $\ell$-adic representation satisfying \eqref{eq:char}.
The following assertions hold.
\begin{enumerate}[(i)]
\item There exists a faithful representation $j_F:\bG_{\rho_\ell,F}\times_F \overline\Q_\ell\hto \GL_{n,\overline\Q_\ell}$
such that 
$$\rho_\ell\simeq [\Gal_K  \stackrel{\pi_F}{\rightarrow} \bG_{\rho_\ell,F}(F)\hto \bG_{\rho_\ell,F}(\overline\Q_\ell)\stackrel{j_F}{\hto}\GL_n(\overline\Q_\ell)].$$
\item The $F$-group $\bG_{\rho_\ell,F}$ is an $F$-model of the algebraic monodromy group $\bG_\ell$ of $\rho_\ell$.
In particular, $\bG_{\rho_\ell,F}$ is reductive and $\pi_0(\bG_{\rho_\ell,F})\simeq \pi_0(\bG_\ell)$.
\item There is a finite extension $F'$ of $F$ such that $\bG_{\rho_\ell,F'}$ is a subgroup scheme of $\GL_{n,F'}$.
\item If $F'\subset \overline\Q_\ell$ is a subfield containing $F$,  
then there is a natural isomorphism $\iota_{F'/F}:\bG_{\rho_\ell,F}\times_FF'\stackrel{\simeq}{\rightarrow} \bG_{\rho_\ell,F'}$
such that $\pi_{F'}=\iota_{F'/F}\circ\pi_F$ and $j_F= j_{F'}\circ\iota_{F'/F}$.
%\item The $F$-group scheme $\bG_{\rho_\ell,F}$ is reductive.
%\item The group scheme $\bG_{\rho_\ell,F}$ is connected if and only if $\bG_{\rho_\ell,\overline \Q_\ell}$ is connected.
\item If $\rho_\ell(\Gal_K)\subset\GL_n(F)$ and $\bG_\ell$ denotes the Zariski closure of $\rho_\ell(\Gal_K)$ in $\GL_{n,F}$,
then  $\bG_{\rho_\ell,F}\simeq \bG_\ell$ and $\Gal_K  \stackrel{\pi_F}{\rightarrow} \bG_{\rho_\ell,F}(F)$ can be identified with
$\Gal_K  \stackrel{\rho_\ell}{\rightarrow} \bG_{\ell}(F)$.
\end{enumerate}
\end{prop}

\begin{proof}
Let $j_F$ be the map $\underline{S}_F^\times\times_F\overline\Q_\ell\hto\GL_{n,\overline\Q_\ell}$ in \eqref{longcompos}.
Since $\rho_\ell$ is equivalent to \eqref{longcompos}, assertion (i) holds by the definition of $\bG_{\rho_\ell,F}$.
Then (ii) and (iii) follow immediately from (i).

If $F'$ is a field such that $F\subset F'\subset\overline\Q_\ell$,
then $D_{\rho_\ell/F'}=D_{\rho_\ell/F}\otimes_F F'$ (by uniqueness of the $F'$-valued determinant law) and 
$\mathrm{Ker}(D_{\rho_\ell/F'})=\mathrm{Ker}(D_{\rho_\ell/F})\otimes_F F'$ (\cite[Lemma 2.8(i)]{Ch14})
so that $S_{F'}=S_F\otimes_F F'$ and $\pi_{F'}=\pi_{F}\otimes_F F'$.
These together with (b-ii) imply (iv).

If $\rho_\ell(\Gal_K)\subset\GL_n(F)$, then $D_{\rho_\ell/F}$ is the same as 
$F[\Gal_K]\stackrel{\rho_\ell}{\rightarrow}M_n(F)\stackrel{\det}{\rightarrow} F$ and $S_F$ is just the image 
of $F[\Gal_K]$ in $M_n(F)$. This implies (v).
%It follows from the characterizing properties of \cite[Theorem 2.16]{Ch14} that $S$, $\prod_i r_i$ and the tuple of $m_i$ is compatible with base change. This implies the compatibility with base extension. It remains to explain (i)--(iv): Concerning (i) observe first that reductivity descends under separable field extensions, and so by the part already shown, it suffices to prove that $\bG_{\rho_\ell,\overline\Q_\ell}$ is reductive. This however follows from the semisimplicity of $\rho_\ell$. Part (ii) follows from the fact that functor $\pi_0$ on group schemes is compatible with base change. Part (iii) can most easily be deduced from the structure theorem on $S$, by choosing for $F'$ a common splitting field for the finitely many central simple algebras $S_i$. 
\end{proof}

Let $\rho_\lambda:\Gal_K\to\GL_n(\overline E_\lambda)$
be a semisimple $E$-rational $\lambda$-adic representation.
Since $\rho_\lambda$ satisfies \eqref{eq:char} with $\overline\Q_\ell\simeq \Elambda$ and $F=E_\lambda$, the
algebraic monodromy group $\bG_{\rho_\lambda,E_\lambda}$ of $\rho_\lambda$ defined over $E_\lambda$ is well-defined.
If $\rho_\lambda(\Gal_K)\subset \GL_n(E_\lambda)$, then Proposition \ref{prop:Glambda-andBC}(v) asserts 
that $\bG_{\rho_\lambda,E_\lambda}$ is the usual algebraic monodromy group in $\GL_{n,E_\lambda}$.
For a semisimple $E$-rational compatible system, we have the following definition.

\begin{defi}\label{def:Glambda-Erat}
Let $\{\rho_\lambda: \Gal_K\stackrel{}{\rightarrow} \GL_n(\Elambda)\}_{\Pi_E}$ be a semisimple $E$-rational compatible system. 
Its associated system of algebraic monodromy groups defined over $E_\lambda$ 
is $\{\bG_\lambda\}_{\Pi_E}$ with $\bG_\lambda:=\bG_{\rho_\lambda,E_\lambda}$.
Denote the morphism $\Gal_K\stackrel{\pi_{E_\lambda}}{\rightarrow}\bG_{\rho_\lambda,E_\lambda}(E_\lambda)$ also by $\rho_\lambda$.
\end{defi}

\subsection{Proof of Theorem \ref{coeffd}}\label{Descend}
Let $\{\rho_\lambda:\Gal_K\to\GL_n(\overline E_\lambda)\}_{\Pi_E}$ be a semisimple $E$-compatible system
with $\{\bG_\lambda\}$ as the system of algebraic monodromy groups.
By Serre \cite{Se81}, 
there is a finite Galois extension $K^{\conn}/K$ that corresponds to 
the kernel of $\Gal_K\stackrel{\rho_\lambda}{\rightarrow} \bG_\lambda/\bG_\lambda^\circ$ for all $\lambda\in\Pi_E$.

By replacing $K$ with $K^{\conn}$, we first deal with the connected case (i.e., $\bG_\lambda$ is connected for all $\lambda$).
According to (b) and the map \eqref{linearcompos} from subsection \ref{ssec:algmon} 
%and the map \eqref{linearcompos} 
(for $F=E_\lambda$ and $\overline\Q_\ell=\overline E_\lambda$), if the number field $E$ is large enough such 
the $F_i$-central simple algebra  $S_i$ is split, i.e., isomorphic to $M_{d_i}(F_i)$ for all $i$, then  using a map as in \eqref{eq:RS} one sees that $\rho_\lambda$ can be defined over $E_\lambda$.
If we can find a large enough $E$ for this to hold for all $\lambda\in\Pi_E$, then Theorem \ref{coeffd}(i) follows.
By applying \cite[$\mathsection7$]{LP92} to the semisimple $E$-compatible system $\{\rho_\lambda\}$ 
(under the connectedness assumption), there exist a finite subset\footnote{The size of such $B$ can be two.}
$B\subset \Sigma_K\backslash S$ such that for each $\lambda\in\Pi_E$,
 $\rho_\lambda$ is unramified at some element of $v_\lambda\in B$ and 
the Frobenius conjugacy class 
$$\rho_\lambda(\mathrm{Frob}_{v_\lambda})\subset\bG_\lambda(\overline E_\lambda)\subset\GL_n(\overline E_\lambda)$$
is \emph{$\Gamma$-regular} \cite[Definition 4.5]{LP92}. Suppose $E$ is large enough 
to contain the roots of $P_v(T):=\det(\rho_\lambda(\mathrm{Frob}_v)-T\cdot\text{Id})\in E[T]$ for all $v\in B$.
Fix an element $t_\lambda\in \rho_\lambda(\mathrm{Frob}_{v_\lambda})$. Then $t_\lambda$ lies in a unique maximal 
torus $\bT_\lambda$ of $\bG_\lambda$ \cite[Proposition 4.7]{LP92}. 
For any $E_\lambda$-embedding $\sigma:F_i\to \overline E_\lambda$, we obtain a map 
$\iota_\sigma: S_i\to S_i\otimes_{F_i} \overline E_\lambda\simeq M_{d_i}(\overline E_\lambda)$ 
 such that the representation
$$\Gal_K\stackrel{\pi_{E_\lambda}}{\rightarrow} S_{E_\lambda}\to S_i\to M_{d_i}(\overline E_\lambda)$$
is an irreducible factor of $\rho_\lambda$ (subsection \ref{ssec:algmon}(b)),
which corresponds to an irreducible factor $W_\lambda$ of the faithful
representation $\bG_\lambda\to\GL_{n,\overline E_\lambda}$.
Since $\bG_\lambda$ is connected reductive, some weight (e.g., the ``highest weight'') of $\bT_\lambda\to \GL_{W_\lambda}$
has multiplicity one. Since the eigenspace decomposition of $t_\lambda$
on the ambient space $\overline E_\lambda^n$
coincides with that of $\bT_\lambda$ \cite[(4.4)]{LP92}, the characteristic polynomial 
of the image of $\mathrm{Frob}_{v_\lambda}$
in $S_i$ splits completely over $F_i$ by our choice of $E$ and has a root of multiplicity one.
The latter condition implies that $S_i=M_{d_i}(F_i)$ by \cite[Theorem 4(a)]{Yu13}.
Since this holds for all $\lambda$ and $F_i/E_\lambda$, we obtain Theorem \ref{coeffd}(i).
Now, we may assume that $\rho_\lambda(\Gal_K)\subset\GL_n(E_\lambda)$ for all $\lambda$.
Hence, we obtain 
$$t_\lambda\in \bT_\lambda(E_\lambda)\subset\bG_\lambda(E_\lambda)\subset\GL_{n}(E_\lambda),$$
where $\bT_\lambda\subset\bG_\lambda$ are subgroups of $\GL_{n,E_\lambda}$.
 Now by the proof of \cite[Proposition 4.7]{LP92}, the torus $\bT_\lambda$ is a subgroup 
of the center of the centralizer of $t_\lambda$ inside $\GL_{n,E_\lambda}$. Because $t_\lambda$ can 
be diagonalized in $\GL_n(E_\lambda)$, by our choice of $E$, its centralizer
in $\GL_{n,E_\lambda}$ is a product $\prod_i \GL_{n_i,E_\lambda}$ with $\sum_in_i=n$, and the center of this centralizer is an $E_\lambda$-split torus. 
%Moreover by the proof of \cite[Proposition 4.7]{LP92}, the torus $\bT_\lambda$ is a subgroup of the center of the centralizer of $t_\lambda$, and h
Now as a subtorus, %of an $E_\lambda$-split torus 
the torus $\bT_\lambda$ is also $E_\lambda$-split, cf.~\cite[\S~12c and e]{Mi17}, and this gives Theorem \ref{coeffd}(ii).
From the splitness of $\bG_\lambda$ Theorem \ref{coeffd}(iii) is immediate,  since for a split reductive group over a field of characteristic zero,
any irreducible representation is absolutely irreducible (by the highest weight theory for such representations); \hbox{cf.~\cite[\S 22.a]{Mi17}}.

Next, we deal with the general (disconnected) case. From above, we assume $E$ is large enough such that 
$\rho_\lambda(\Gal_{K^{\conn}})\subset\GL_n(E_\lambda)$ and every irreducible factor of $\rho_\lambda|_{\Gal_{K^{\conn}}}$
is absolutely irreducible for all $\lambda$. By Frobenius reciprocity,
\begin{equation}\label{ind}
\{\Phi_\lambda:=\mathrm{Ind}_{K^{\conn}}^K (\rho_\lambda|_{\Gal_{K^{\conn}}}):\Gal_K\to\GL_{n[K^{\conn}:K]}(E_\lambda)\}_{\Pi_E}
\end{equation}
is a semisimple strongly $E$-compatible system that, after base change to $\overline E_\lambda$, contains $\{\rho_\lambda\}_{\Pi_E}$ 
as a subsystem. To finish, it suffices to show that every irreducible factor $\phi_\lambda$ 
of every $\Phi_\lambda$ in \eqref{ind} is absolutely irreducible
if $E$ is large enough. By our assumption and $K^{\conn}/K$ being Galois, every irreducible factor of 
$\Phi_\lambda|_{\Gal_{K^{\conn}}}$ is absolutely irreducible.
We formulate the following purely representation theoretic results based on \cite{Cl37};  see also \cite[\S~49--51]{CR62}.

\begin{prop}\label{Clifford}
Let $F$ be a field, $\phi:G\to\GL(V)$ be a finite dimensional irreducible $F$-representation of a group $G$, 
and $H$ a normal subgroup of $G$ of finite index.
The following assertions hold.
\begin{enumerate}[(i)]
\item The restriction $\phi|_H$ is semisimple and its irreducible decomposition is of the form $V=\bigoplus_{i=1}^m W_i^{\oplus k}$
for some $m,k\in\N$, where $W_i$ and $W_j$ are non-isomorphic if $i\neq j$. Moreover, the dimensions of $W_i$ for all $i$ are equal to $d\in\N$.
\item Let $G_1\leq G$ be the stabilizer of the subspace $V_1:=W_1^{\oplus k}\leq V$. Then the index $[G:G_1]=m$ and 
the natural representation $\phi_1:G_1\to \GL(V_1)$ is irreducible.
\item Suppose the (irreducible) $H$-representation $W_1$ is absolutely irreducible.
Then $\phi_1$ is the tensor product $\alpha_1\otimes\beta_1$ of 
two irreducible projective $F$-representations\footnote{A projective representation
$G\to M_k(F)$ is equivalent to a group homomorphism $G\to\PGL_k(F)$.} 
$\alpha_1:G_1\to M_k(F)$ and $\beta_1:G_1\to M_d(F)$ of $G_1$ such that $\alpha_1$ is a projective representation 
of the quotient group $G_1/H$
and $\beta_1|_H$ is the true representation $W_1$.
\item Suppose the $H$-representation $W_i$ is absolutely irreducible for all $1\leq i\leq m$.
If any subgroup $G'$ of $G$ containing $H$ and any irreducible projective $F$-representation of $G'/H$
is absolutely irreducible, then the irreducible representation $\phi$ of $G$ is absolutely irreducible. 
\end{enumerate}
\end{prop}

\begin{proof}
Assertions (i),(ii) follow from \cite[Theorem 2]{Cl37} and its proof.
Assertion (iii) is just \cite[Theorem 3]{Cl37}, where one only requires $W_1$ to be absolutely irreducible so that 
the commutant $\End_H(W_1)=F$.

Regarding (iv), let us first recall some basic results on projective representations, e.g.~\cite[\S 24--25]{Hu67}. 
To every projective representation of $G$ on a vector space $U$, one can in the usual way define the algebra of $G$-equivariant
endomorphisms; the $2$-cocycle of $G$ valued in $F^\times$ that is needed to define the action of $G$ on $U$, 
does not interfere, because the center acts trivially under conjugation. Moreover, as is the case for ordinary representations, 
a projective representation is absolutely irreducible over $F$ if and only if it is irreducible and $\End_G(U)=F$. %To every projective representation of $G$, one can associate a $1$-cocycle in $Z^2(G,F^\times)$, and the cocycles of isomorphic projective representations differ by a $1$-coboundary.  representations

We can now argue as follows, where below $U$ denotes the space underlying the representation $\alpha_1$ from (iii):\footnote{Found in math overflow {\tt https://mathoverflow.net/questions/208731/irreducibility-of-the-tensor-product-of-two-finite-dimensional-irreducib} 
%{\color{violet} should we quote this; cf. comment of user 74720}
}
\begin{eqnarray*}
\End_G(V)&=&\End(W_1\otimes U)^G \ = \ (\End(W_1)\otimes\End(U))^G \\ 
&\stackrel{\mathrm{(1)}}= & \ (\End(W_1)^H\otimes\End(U))^{G/H} \ = (F\otimes\End(U))^{G/H} \  \stackrel{\mathrm{(2)}}= \ F,
\end{eqnarray*}
where for (1) we use our hypothesis that $W_1$ is absolutely irreducible and that $H$ acts trivially on $U$, and for (2) that the projective representation $\beta_1$ is an absolutely irreducible projective $G/H$-representation.
%Let $G_i$ be the stabilizer of $V_i:=W_i^{\oplus k}\leq V$ for all $i$.
%By (iii) and the assumption of (iv), one obtains 
%$$\End_G(V)\subset \prod_{i=1}^m \End_{G_i}(V_i)=\prod_{i=1}^m F= F^{\oplus m}.$$
%It follows that $\End_G(V)=F$ since it is an $F$-division algebra.
Hence $\phi$ is absolutely irreducible. 
\end{proof}

Let $\phi_\lambda$ be an irreducible  subrepresentation of $\Phi_\lambda$.
We would like to use Proposition \ref{Clifford} to show that $\phi_\lambda$ is absolutely irreducible.
Take $\phi=\phi_\lambda$, $G=\Gal_K$, $H=\Gal_{K^{\conn}}$. If $E$ is large enough, then 
the assumptions of Proposition \ref{Clifford}(iv) are satisfied.
Indeed, absolute irreducibility of $\phi_\lambda|_{\Gal_{K^{\conn}}}$ is already assumed. Moreover 
the finite group $G/H$ has finitely many subgroups $Q$ and each $Q$ has finitely many 
irreducible projective $\C$-representations (up to isomorphism)
which can be defined over a large enough number field $E$ (and thus over $E_\lambda$);  see \cite[Theorem 11.17]{Is76}.
This completes the proof.

\section{Abelian $\ell$-adic representations}\label{s4}
We recall some results frome \cite{Se98} and \cite{BH25} that are fundamental to the present article.
Let $K$ and $E$ be number fields. 
If $\phi_\ell$ is an abelian $\ell$-adic representation of $K$, then it factors through the Galois group $\Gal_K^{\ab}$ of 
the maximal abelian extension of $K$. Thus, we can identify $\phi_\ell$ with $\phi_\ell:\Gal_K^{\ab}\to\GL_n(\overline \Q_\ell)$.

\subsection{Local algebraicity}\label{s4.1}
Let $\bT_K:=\mathrm{Res}_{K/\Q}\mathbb{G}_{m,K}$ be the $[K:\Q]$-dimensional $\Q$-torus 
given by Weil restriction of scalars. %For any commutative $\Q$-algebra $A$, the group of $A$-points satisfies $$\bT_K(A)=(K\otimes_\Q A)^\times.$$
Let $\A_K^\times$ be the group of id\`eles of $K$ and $i_\ell$ be the natural composition
\begin{equation*}\label{localArtin}
i_\ell:\bT_K(\Q_\ell)=(K\otimes_\Q\Q_\ell)^\times=\prod_{v|\ell}K_v^\times\to \A_K^\times/K^\times \stackrel{\mathrm{Art}_K}{\longrightarrow}\Gal_K^{\text{ab}},
\end{equation*}
where $\mathrm{Art}_K$ is the Artin reciprocity map.
An abelian semisimple $\ell$-adic representation 
\begin{equation}\label{phi}
\phi_\ell:\Gal_K^{\ab}\to\GL_n(\overline\Q_\ell)
\end{equation}
of $K$ is called \emph{locally algebraic} \cite[Chap. III]{Se98}
if there exists an $\overline\Q_\ell$-algebraic morphism
$$r_\ell:\bT_K\times_\Q\overline\Q_\ell \to \GL_{n,\overline\Q_\ell}$$
such that for all $x$ in some neighborhood of $1$ in the $\ell$-adic Lie group 
$\bT_K(\Q_\ell)$:
\begin{equation*}\label{eq:LocAlgMap}
\phi_\ell\circ i_\ell(x)=r_\ell(x^{-1}).
\end{equation*}

The local algebraicity of $\phi_\ell$ is equivalent to, respectively, the local algebraicity of $\phi_\ell|_{\Gal_L}$ (for any $L/K$ finite)
and the local representations of $\phi_\ell$ at all $v\in S_\ell$ being Hodge-Tate (or de Rham) \cite[Chap. III]{Se98}.
The following equivalence is crucial.

\begin{thm}(Serre-Waldschmidt, see \cite{He82}) \label{SW}
An abelian semisimple $\ell$-adic representation $\phi_\ell$ is locally algebraic if and only if it is $E$-rational for some 
number field $E\subset\overline\Q_\ell$.
\end{thm}

\subsection{Serre groups $\bS_\mathfrak{m}$ and representations}\label{s4.2}
A \emph{modulus} $\mathfrak m$ of $K$ is a function $\mathfrak m:\Sigma_K \to \Z_{\geq 0}$ such that 
the \emph{support} $\mathrm{Supp}(\mathfrak m):=\mathfrak m^{-1}(\Z_{>1})$ is finite.
By \cite[Chap. II]{Se98}, attached to a modulus $\mathfrak m$ is a pair $(\bS_{\mathfrak m}, \{\epsilon_\ell\})$,
where $\bS_\mathfrak{m}$  is a
$\Q$-diagonalizable group  (called the Serre group with modulus $\mathfrak m$) and
\begin{equation}\label{epsilon}
\{\epsilon_\ell:\Gal_K\to \bS_\mathfrak{m}(\Q_\ell):~\ell\in\Sigma_\Q\}
\end{equation}
is a family of abelian $\ell$-adic representations of $K$
such that $\epsilon_\ell(\Gal_K)$ is Zariski dense in $\bS_\mathfrak{m}\times\Q_\ell$ for all $\ell$,
and
if $\phi:\bS_\mathfrak{m}\times E\to\GL_{n,E}$
is an $E$-morphism, then 
\begin{equation}\label{phics}
\{\phi_\lambda:\Gal_K\stackrel{\epsilon_\ell}{\rightarrow} \bS_\mathfrak{m}(\Q_\ell) \hookrightarrow\bS_\mathfrak{m}(E_\lambda)
\stackrel{\phi}{\rightarrow}  \GL_n(E_\lambda):~  \lambda\in\Sigma_E\}
\end{equation}
is an abelian semisimple  strongly $E$-rational compatible system of $K$ with the finite $S=\mathrm{Supp}(\mathfrak m)$.
The pairs $(\bS_{\mathfrak m}, \{\epsilon_\ell\})$ for all moduli $\mathfrak m$ of $K$
form a projective system (with respect to the arrows $\mathfrak m\leftarrow \mathfrak m'$ whenever $\mathfrak m\leq \mathfrak m'$) 
\cite[Chap II, $\mathsection2.2$, Remark]{Se98}.

\begin{thm}\label{Serretolocalg}\cite[Chap. III, $\mathsection 2.3$, Theorem 1]{Se98}
The semisimple abelian $\lambda$-adic representation $\phi_\lambda=\phi\circ\epsilon_\ell$ 
(attached to $\phi:\bS_\mathfrak{m}\times E\to\GL_{n,E}$) is locally algebraic for all $\lambda$.
\end{thm}

We present some extension results.

\begin{prop}\label{extend}
Let $\phi_\ell:\Gal_K\to\GL_n(\overline \Q_\ell)$ be an abelian semisimple $\ell$-adic representation.
If $\phi_\ell$ is $E$-rational for some number field $E\subset\overline\Q_\ell$, 
then there exists an $E$-morphism $\phi:\bS_\mathfrak{m}\times E\to\GL_{n,E}$
of some Serre group such that $\phi_\ell$ is equivalent to a member of the  strongly $E$-rational compatible system \eqref{phics}.
\end{prop}

\begin{proof}
Since $\phi_\ell$ is $E$-rational, it is locally algebraic by Theorem \ref{SW}.
Then there exist a modulus $\mathfrak m$ and a 
$\overline\Q_\ell$-morphism $\phi:\bS_\mathfrak{m}\times \overline\Q_\ell\to\GL_{n,\overline\Q_\ell}$ 
such that 
\begin{equation}\label{compose}
\phi_\ell\simeq \phi\circ \epsilon_\ell.
\end{equation}
by \cite[Chap. III, $\mathsection2.3$, Theorem 2]{Se98}.
The $\overline\Q_\ell$-morphism $\phi$ is defined over $E$ by the $E$-rationality and \cite[Chap. II, $\mathsection2.4$, Proposition 2]{Se98}.
Hence, we obtain the $E$-rational compatible system \eqref{phics} and
we are done by \eqref{compose}.
\end{proof}

The following results are  immediate consequences of Proposition~\ref{extend} and Theorem~\ref{unique};
the second one after a finite base change.

\begin{cor}\label{cor:abelian-strongly}
Any abelian semisimple $E$-rational compatible system is isomorphic to a strongly $E$-rational compatible system.
\end{cor}

\begin{cor}\label{pabisindependent}
Let $\{\rho_\lambda\}_{\Pi_E}$ be a semisimple $E$-rational compatible system.
If there exists $\lambda_0\in\Pi_E$ such that $\rho_{\lambda_0}$ is potentially abelian,
then $\rho_\lambda$ is potentially abelian for all $\lambda\in\Pi_E$.
\end{cor}

So from now on, for any abelian semisimple $E$-rational compatible system we always consider a strongly $E$-rational representative in its isomorphism class.

\begin{remark}
Corollary~\ref{cor:abelian-strongly} is in line with the classical result  that a complex representation of a finite abelian group $G$ with field of traces $E$ has a model over $E$ -- because $\C[G]$ is abelian and so all Schur indices are trivial; cf.~\cite[$\mathsection12.2$]{Se77}.
\end{remark}

\begin{defi}
For an $E$-morphism $ \phi:\bS_\mathfrak{m}\times E\to\GL_{n,E}$, 
denote by $\bG_\phi$ the closed subgroup $\phi(\bS_\mathfrak{m}\times E)$ of $\GL_{n,E}$. %and by $\iota_\phi$ the canonical inclusion $\bG_\phi\to\GL_{n,E}$.}
\end{defi}

\begin{thm}\label{converse}
Let $\{\phi_\lambda:\Gal_K\to\GL_n(E_\lambda)\}_{\Sigma_E}$ 
be an abelian semisimple $E$-rational compatible system
with $\{\bG_{\phi_\lambda}\}_{\Sigma_E}$ as the system of algebraic monodromy groups. 
There exists an $E$-morphism $\phi:\bS_\mathfrak{m}\times E\to\GL_{n,E}$
of some Serre group such that $\{\phi_\lambda\}$ is the compatible system
attached to $\phi$ as \eqref{phics}. In particular, 
\begin{equation}\label{abindep}
[\bG_\phi\subset\GL_{n,E}]\times E_\lambda \simeq [\bG_{\phi_\lambda}\subset \GL_{n,E_\lambda}]
\hspace{.3in}\text{for all} ~\lambda\in\Sigma_E.
\end{equation}
\end{thm}

\begin{proof}
By Proposition \ref{extend} and Theorem \ref{unique}, the compatible system is attached to 
a Serre group representation $\phi:\bS_\mathfrak{m}\times E\to\GL_{n,E}$.
Finally, \eqref{abindep} follows from the fact that 
$\epsilon_\ell(\Gal_K)$ (in \eqref{epsilon}) is Zariski dense in $\bS_\mathfrak{m}\times\Q_\ell$ for all $\ell$.
\end{proof}

\begin{prop}\label{extendsys}
Let $\{\phi_\lambda:\Gal_K\to\GL_k(E_\lambda)\}_{\Sigma_E}$ and $\{\psi_\lambda:\Gal_K\to\GL_h(E_\lambda)\}_{\Sigma_E}$
be two abelian semisimple $E$-rational compatible systems with $\{\bG_{\phi_\lambda}\}_{\Sigma_E}$ and $\{\bG_{\psi_\lambda}\}_{\Sigma_E}$ as their systems of algebraic monodromy groups. Let $ \phi:\bS_\mathfrak{m}\times E\to\GL_{k,E}$ and $ \psi:\bS_\mathfrak{m}\times E\to\GL_{h,E}$ be the morphisms from Theorem~\ref{converse}, with $\mathfrak{m}$ as the supremum of the individual moduli. Suppose for some $\lambda_0\in \Sigma_E$ there is an $E_{\lambda_0}$-morphism $\alpha_{\lambda_0}:\bG_{\phi_{\lambda_0}}\to \bG_{\psi_{\lambda_0}}$, such that
\begin{equation}\label{0map}
\psi_{\lambda_0} \simeq
[\Gal_K\stackrel{\phi_{\lambda_0}}{\rightarrow}\bG_{\phi_{\lambda_0}}(E_{\lambda_0})
\stackrel{\alpha_{\lambda_0}}{\rightarrow} \bG_{\psi_{\lambda_0}}(E_{\lambda_0})\subset\GL_h(E_{\lambda_0})].
\end{equation}
Then there is a unique $E$-morphism $\alpha:\bG_\phi\to\bG_\psi$ such that
\begin{equation}\label{satisfy}
\alpha\circ(\phi:\bS_\mathfrak{m}\times E \to\bG_\phi)=(\psi:\bS_\mathfrak{m}\times E \to \bG_\psi).
\end{equation}
 Moreover, the following assertions hold for $\alpha$.
\begin{enumerate}[(i)]
\item For all $\lambda\in\Sigma_E$ the $E_\lambda$-morphism $\alpha_\lambda:=\alpha\otimes E_\lambda:\bG_{\phi_{\lambda}}\to \bG_{\psi_{\lambda}}$ satisfies
\begin{equation}\label{genericmap}
\psi_{\lambda}\simeq
[\Gal_K\stackrel{\phi_{\lambda}}{\rightarrow}\bG_{\phi_{\lambda}}(E_\lambda)
\stackrel{\alpha_{\lambda}}{\rightarrow} \bG_{\psi_{\lambda}}(E_\lambda)\subset\GL_h(E_\lambda)].
\end{equation}
\item \label{item:alphaisom} If $\alpha_{\lambda_0}$ is surjective (or injective, or bijective), then so is $\alpha$, and hence all $\alpha_\lambda$, $\lambda\in \Sigma_E$.
\end{enumerate}
%Let $\{\phi_\lambda:\Gal_K\to\GL_k(E_\lambda)\}_{\Sigma_E}$ and $\{\psi_\lambda:\Gal_K\to\GL_h(E_\lambda)\}_{\Sigma_E}$
%be two abelian semisimple $E$-rational compatible systems 
%with $\{\bG_{\phi_\lambda}\}_{\Sigma_E}$ and $\{\bG_{\psi_\lambda}\}_{\Sigma_E}$ as the systems of algebraic monodromy groups.
%If $\alpha_{\lambda_0}:\bG_{\phi_{\lambda_0}}\to \bG_{\psi_{\lambda_0}}$ is an $E_{\lambda_0}$-morphism
%such that 
%\begin{equation}\label{0map}
%\psi_{\lambda_0} \simeq
%[\Gal_K\stackrel{\phi_{\lambda_0}}{\rightarrow}\bG_{\phi_{\lambda_0}}(E_\lambda)
%\stackrel{\alpha_{\lambda_0}}{\rightarrow} \bG_{\psi_{\lambda_0}}(E_\lambda)\subset\GL_h(E_\lambda)],
%\end{equation}
%then there exists a family $\{\alpha_{\lambda}:\bG_{\phi_{\lambda}}\to \bG_{\psi_{\lambda}}\}_{\Sigma_E}$
%of $E_\lambda$-morphisms such that for all $\lambda\in\Sigma_E$, 
%\begin{equation}\label{genericmap}
%\psi_{\lambda}\simeq
%[\Gal_K\stackrel{\phi_{\lambda}}{\rightarrow}\bG_{\phi_{\lambda}}(E_\lambda)
%\stackrel{\alpha_{\lambda}}{\rightarrow} \bG_{\psi_{\lambda}}(E_\lambda)\subset\GL_h(E_\lambda)].
%\end{equation}
\end{prop}

\begin{proof}
Note that $\phi_\lambda=(\phi\otimes E_\lambda)\circ \epsilon_\ell$ and  $\psi_\lambda=(\psi\otimes E_\lambda)\circ \epsilon_\ell$. 
Because $\epsilon_\ell$ has Zariski dense image in $\bS_{\mathfrak{m}}(E_\lambda)$, it follows from \eqref{0map},
that $\psi(\mathrm{Ker}(\phi))\times E_{\lambda_0}$ is trivial. This implies that the $E$-morphism 
$\psi:\bS_{\mathfrak m}\to \bG_\psi$ factors through $\phi:\bS_{\mathfrak m}\to \bG_\phi$, 
giving the existence of
$\alpha:\bG_\phi\to\bG_\psi$ such that \eqref{satisfy} holds. The uniqueness 
is clear from the surjectivity of $\phi:\bS_\mathfrak{m} \to\bG_\phi$. Finally, \eqref{genericmap} follows from the construction of compatible 
systems from Serre group representations, and \eqref{item:alphaisom} from the faithfully flatness of $E\to E_\lambda$.
%By Theorem \ref{converse} and 
%$\{(\bS_{\mathfrak m}, \{\epsilon_\ell\}):~\mathfrak m~\text{modulus of}~K \}$
%being a projective system, we may assume the $E$-rational compatible systems 
%$\{\phi_\lambda\}$ and $\{\psi_\lambda\}$
%are attached to Serre group representations 
%$$\phi:\bS_{\mathfrak m}\to\GL_{k,E}\hspace{.3in} \text{and}\hspace{.3in} \psi:\bS_{\mathfrak m}\to\GL_{h,E}$$ 
%for a common modulus $\mathfrak m$. Note that $\mathrm{Im}(\phi)\times E_\lambda=\bG_{\phi_{\lambda}}$
%and $\mathrm{Im}(\psi)\times E_\lambda=\bG_{\psi_{\lambda}}$ for all $\lambda$.
%Since $\psi(\mathrm{Ker}(\phi))\times E_{\lambda_0}$ is trivial by \eqref{0map},
%the $E$-morphism $\psi:\bS_{\mathfrak m}\to \mathrm{Im}(\psi)$ factors through $\phi:\bS_{\mathfrak m}\to \mathrm{Im}(\phi)$.
%Thus, there is a $E$-morphism $\alpha:\mathrm{Im}(\phi)\to \mathrm{Im}(\psi)$ such that 
%$$\psi=\alpha\circ\phi,$$
%and \eqref{0map} follows from the construction of $\alpha$.
%We obtain \eqref{genericmap} for the $\alpha_\lambda:= \alpha\times E_\lambda$ %for all $\lambda$
%by the construction of compatible systems from Serre group representations. 
\end{proof}

\subsection{Weak abelian direct summands}\label{s4.3}
Let $\rho_\ell:\Gal_K\to\GL_n(\overline \Q_\ell)$ and $\psi_\ell:\Gal_K\to\GL_m(\overline \Q_\ell)$
be semisimple $\ell$-adic representations of $K$ that are unramified at almost all $v\in\Sigma_K$.
Denote by $S_{\rho_\ell},S_{\psi_\ell}\subset \Sigma_K$ the set of ramified places of $\rho_\ell$ and $\psi_\ell$, respectively. 
We say that $\psi_\ell$ is a \emph{weak direct summand} of $\rho_\ell$ 
(or $\psi_\ell$ \emph{weakly divides} $\rho_\ell$) if the set
\begin{equation}\label{def:Places-psi-wd-rho}
S_{\psi_\ell\mid\rho_\ell}:=\{v\in\Sigma_K\setminus (S_{\rho_\ell}\cup S_{\psi_\ell}): 
\det(\psi_\ell(\Frob_v)-T\cdot\text{Id}) ~\mathrm{\,divides~\,}\det(\rho_\ell(\Frob_v)-T\cdot\text{Id})\} 
\end{equation}
is of Dirichlet density one. If $\psi_\ell$ is abelian and weakly divides $\rho_\ell$, 
we say that $\psi_\ell$ is a \emph{weak abelian direct summand} of $\rho_\ell$ \cite[Definition 2.3]{BH25}.
For example, the (abelian) subrepresentations of $\rho_\ell$ are weak (abelian) direct summands of $\rho_\ell$.

We state some results from \cite{BH25} on weak abelian direct summands of (semisimple) 
$E$-rational $\ell$-adic representations and $E$-rational compatible systems.

\begin{thm}\label{BH1}\cite[Theorem 1.1]{BH25}
Let $K$ and $E\subset\overline\Q_\ell$ be number fields, and $\rho_\ell:\Gal_K\to\GL_n(\overline\Q_\ell)$
be a semisimple $E$-rational $\ell$-adic representation of $K$. 
If $\psi_\ell$ is a weak abelian direct summand of $\rho_\ell$, then it is locally algebraic, 
and thus its local representations at places above $\ell$ are de Rham. 
\end{thm}

Given a semisimple $E$-rational compatible system of a number field $K$ indexed by $\Pi_E$
\begin{equation}\label{cs2}
\{\rho_{\lambda}:\Gal_K\to\GL_n( \Elambda)\}_{\Pi_E},
\end{equation}
the connectedness and the \emph{formal bi-character} (see e.g., \cite[$\mathsection2.3$]{Hu23} for definition) 
of the algebraic monodromy group $\bG_\lambda\subset\GL_{n,E_\lambda}$
are both independent of $\lambda$ \cite{Se81},\cite[Theorem 3.19, Remark 3.22]{Hu13}\footnote{
The statements are stated for $\Q$-compatible systems indexed by $\Sigma_\Q$ but the arguments
actually work for our more general setting since they involve only the compatibility 
of $\rho_{\lambda_1}$ and $\rho_{\lambda_2}$ for any pair $\lambda_1,\lambda_2\in\Sigma_E$.}.
Thus, we define the following for the compatible system \eqref{cs2}.

\begin{itemize}
\item We say that \eqref{cs2} is \emph{connected} if for some $\lambda$ (equivalently, for all $\lambda$) $\bG_\lambda$ is connected.
\item If $\bT_\lambda^{\ss}$ is a maximal torus of the semisimple group $[\bG_\lambda^\circ,\bG_\lambda^\circ]$, then
the multiplicity $n_0$ of the zero weight of the faithful representation $\bT_{\lambda,\overline\Q_\ell}^{\ss}\hto\GL_{n,,\overline\Q_\ell}$
is independent of $\lambda$.
\end{itemize}

If \eqref{cs2} is connected, for each $\lambda\in\Pi_E$ there is a maximal weak abelian direct summand 
$$\rho_\lambda^{\wab}:\Gal_K\to \GL_{n_0}(\overline E_\lambda)$$
of $\rho_\lambda$ (that is $n_0$-dimensional\footnote{If $n_0=0$, then $\rho_\lambda$
has no weak abelian direct summands.}), called the \emph{weak abelian part} of $\rho_\lambda$ \cite[Proposition 2.9]{BH25}.
In fact, the following holds:%family $\{\rho_\lambda^{\wab}\}_{\Pi_E}$ can be extended to an $E'$-compatible system, 

\begin{thm}\label{BH2}\cite[Theorem 2.20]{BH25}\footnote{The statement assumes $\Pi_E=\Sigma_E$ 
but the argument there works for any subset $\Pi_E\subset\Sigma_E$.}
Suppose the semisimple $E$-rational compatible system $\{\rho_{\lambda}\}_{\Pi_E}$ is connected with $n_0>0$.
Then there is an $n_0$-dimensional semisimple abelian  strongly $E'$-rational compatible system $\{\rho_{\lambda'}^{\wab}\}_{\Sigma_{E'}}$ 
for some finite extension $E'/E$ such that if $\lambda'\in\Sigma_{E'}$ divides $\lambda\in\Pi_E$, then
$\rho_{\lambda'}^{\wab}$ and $\rho_\lambda^{\wab}$ are equivalent.
\end{thm}

\section{CM motives and potentially abelian representations}\label{s5}

Because of the abelian theory in subsections \ref{s4.1}-\ref{s4.2} 
(in particular Serre-Waldschmidt's  Theorem~\ref{SW}),
%Theorem given in Theorem~\ref{SW} and the concept of being locally \textcolor[rgb]{0,0,1}{algebraic}, 
the Serre group can be regarded as the key tool to understand $E$-rational semisimple abelian compatible systems of number fields $K$. In this section, we explain that for $E$-rational semisimple potentially abelian compatible systems, the Taniyama group, a certain motivic Galois group, plays an analogous role. Following Deligne et al., this is built on the theory of CM motives \cite[II--IV]{DMOS82},\cite[Chap. I]{Sc88}. We recall some of the foundational results and then study the $\lambda$-independence of algebraic monodromy groups of semisimple potentially abelian $E$-compatible systems (Theorem \ref{pabindep}).

%In this section we study $\lambda$-independence of algebraic monodromy groups of semisimple potentially 
%abelian $E$-compatible systems (Theorem \ref{pabindep})
%by the theory of CM motives \cite[II--IV]{DMOS82},\cite[Chap. I]{Sc88}.
Let $K$ and $E$ be number fields, and fix an embedding $\sigma:K\to\C$. 

\subsection{Motivic Galois group $\bM_K^{\CM}$}\label{s5.1}
The category of motives $\mathcal M_K$ (for absolute Hodge cycles) over $K$ 
 with the $\otimes$-functor $H_\sigma: \mathcal M_K \to \text{Vec}_\Q$ (the category of $\Q$-vector spaces)
is a \emph{semisimple Tannakian category}. The category $\mathcal{CM}_K$ of \emph{CM motives over $K$}
is the Tannakian subcategory of $\mathcal M_K$ generated by \emph{Artin motives} over $K$
and the motives of abelian varieties over $K$ that admit \emph{complex multiplications} over $\overline K$.
According to Tannakian theory, there is a pro-reductive group $\bM_K^{\CM}$ over $\Q$,
called the \emph{motivic Galois group} of $\mathcal{CM}_K$, such that 
$\mathcal{CM}_K$ is equivalent to 
the Tannakian category $\text{Rep}_\Q(\bM_K^{\CM})$ of the finite dimensional 
$\Q$-representations of $\bM_K^{\CM}$ \cite[II]{DMOS82},\cite[Chap. I, $\mathsection2$]{Sc88}.

By identifying $\bM_\Q^{\CM}$ with the \emph{Taniyama group},
the following assertions about $\bM_K^{\CM}$ are obtained  \cite[III, IV]{DMOS82}.

\begin{thm}\label{Taniyama}
There is an epimorphism $e: \bM_\Q^{\CM}\twoheadrightarrow\Gal_\Q$ of pro-reductive groups over $\Q$ such that the
following assertions hold.
\begin{enumerate}[(i)]
\item The kernel of the epimorphism $e$ is a pro-torus, and  one has $e^{-1}(\Gal_K)=\bM_K^{\CM}$.% for the the preimage of the subgroup $\Gal_K$ of $\Gal_\Q$.}%, where $\Gal_K=\Gal(\overline\Q/K)$ is the subgroup of $\Gal_\Q$.
\item There is a continuous section $s:\Gal_\Q\to \bM_\Q^{\CM}(\A_\Q^{f})$ of  $e$ %the epimorphism, 
where $\A_\Q^{f}$ denotes the ring of finite ad\`eles of $\Q$.
 If $s_\ell:\Gal_\Q\to \bM_\Q^{\CM}(\Q_\ell)$ denotes the $\ell$-component of $s$,
the image $s_\ell(\Gal_\Q)$ is Zariski dense in $\bM_\Q^{\CM}$.
\item The continuous morphisms $s|_{\Gal_K}:\Gal_K\to \bM_K^{\CM}(\A_\Q^{f})$ and $s_\ell|_{\Gal_K}:\Gal_K\to \bM_K^{\CM}(\Q_\ell)$
are well-defined, and the image $s_\ell(\Gal_K)$ is Zariski dense in $\bM_K^{\CM}$. 
\item Let $\phi:\bM_K^{\CM}\to\GL_{n,\Q}$ be a $\Q$-representation and 
$M\in \mathcal{CM}_K$ be the object corresponding to $\phi$. 
Then the $\ell$-adic realization $H_\ell(M)$ of $M$
is equivalent the (semisimple) $\ell$-adic representation
$$\phi\circ s_\ell:\Gal_K\to \bM_K^{\CM}(\Q_\ell)\to \GL_n(\Q_\ell).$$
\end{enumerate}
\end{thm}

\begin{proof}
Assertions (i) and (ii) are obtained in \cite[IV]{DMOS82}; (i) can also be seen in \cite[Chap. I, $\mathsection6.5$]{Sc88}. 
Then (iii) follows from (i) and (ii).
Assertion (iv) is stated in \cite[IV, p.265]{DMOS82}.
\end{proof}

\subsection{Representations of $\bM_K^{\CM}$ with $E$-coefficients}\label{s5.2}
There is also the useful Tannakian category $\mathcal{CM}_K(E)$ of \emph{CM motives over $K$ with coefficients in $E$}, 
whose objects are pairs $(M,\theta)$, 
with $M$ an object in $\mathcal{CM}_K$ and $\theta: E\to \End(M)$ an embedding of $\Q$-algebras. 
It is equivalent to 
the Tannakian category 
$\text{Rep}_E(\bM_K^{\CM}\times E)$ of the finite dimensional $E$-representations  of $\bM_K^{\CM}\times E$
such that 
\begin{equation}\label{cd}
\begin{aligned}
\xymatrix{
 \mathcal{CM}_K(E) \ar[d]_{\simeq} \ar[r]^-{\text{Forget}} & \mathcal{CM}_K \ar[d]^{\simeq}\\
\text{Rep}_E(\bM_K^{\CM}\times E) \ar[r]^-{\mathrm{Res}_{E/\Q}} & \text{Rep}_\Q(\bM_K^{\CM}) 
}
\end{aligned}
\end{equation}
is commutative, where the top map sends $(M,\theta)$ to $M$ and 
the bottom map sends an $E$-representation $\phi$ to the $\Q$-representation $\mathrm{Res}_{E/\Q}(\phi)$ in \eqref{eq:RS}
\cite[Chap. I, $\mathsection3$]{Sc88}.

Given a correspondence
\begin{equation}\label{rhoE}
(\phi:\bM_K^{\CM}\times E\to\GL_{k,E})\hspace{.2in} \leftrightarrow\hspace{.2in} (M,\theta),
\end{equation}
we obtain two families of $\lambda$-adic representations of $K$ as follows.
Firstly, composing $\phi$ with $\{s_\ell\}$ in Theorem \ref{Taniyama}(iii) gives the family
\begin{equation}\label{1stcs}
\{\Gal_K\stackrel{s_\ell}{\rightarrow} 
\bM_K^{\CM}(\Q_\ell)\stackrel{\phi\otimes E_\lambda}{\longrightarrow} \GL_k(E_\lambda)\}_{\lambda\in\Sigma_E}.
\end{equation}
Secondly, since the $\ell$-adic realization $H_\ell(M)$  
is a free $E\otimes\Q_\ell=\prod_{\lambda|\ell}E_\lambda$-module, it
admits a decomposition 
\begin{equation}\label{lambdadecompose}
H_\ell(M)=\prod_{\lambda|\ell}[\phi_\lambda:\Gal_K\to\GL_k(E_\lambda)]
\end{equation}
for every $\ell$, which gives the family
\begin{equation}\label{2ndcs}
\{\phi_\lambda:\Gal_K\to\GL_k(E_\lambda)\}_{\lambda\in\Sigma_E}.
\end{equation}
According to \cite[Chap. I, Corollary 6.5.7]{Sc88}, \eqref{2ndcs} is an $E$-rational compatible system of $K$.

\begin{thm}\label{rholambda}
The semisimple potentially abelian $\lambda$-adic representation $(\phi\otimes E_\lambda)\circ s_\ell$
in \eqref{1stcs} is equivalent to $\phi_\lambda$ in \eqref{2ndcs} for all $\lambda$.
 In particular, \eqref{1stcs}
is a semisimple potentially abelian  strongly $E$-rational compatible system of $K$.
\end{thm}

\begin{proof}
By \eqref{cd}, \eqref{rhoE}, and Theorem \ref{Taniyama}(iv),
the $\ell$-adic realization $H_\ell(M)$ is  $\text{Res}_{E/\Q}(\phi)\circ s_\ell$. 
Then \eqref{lambdadecompose} and Proposition \ref{see} give the equivalence
$$\prod_{\lambda|\ell}\phi_\lambda=\text{Res}_{E/\Q}(\phi)\circ s_\ell\simeq \prod_{\lambda|\ell} (\phi\otimes E_\lambda) \circ s_\ell.$$
that is compatible with the $E\otimes\Q_\ell$-action.
Therefore, $\phi_\lambda\simeq (\phi\otimes E_\lambda) \circ s_\ell$ and we are done.
\end{proof}

\subsection{Potentially locally algebraic representations}\label{s5.3}
\begin{defi}\label{def:plocalg}
An  $\ell$-adic representation  $\phi_\ell:\Gal_K\to\GL_k(\overline\Q_\ell)$ of a number field $K$ is 
said to be \emph{potentially locally algebraic} if
there exists a finite extension $L/K$
such that the restriction $\phi_\ell|_{\Gal_L}$ is (i) semisimple abelian and (ii) locally algebraic (subsection \ref{s4.1}). 
\end{defi}
%{\color{blue} 
%We remark that any $\ell$-adic potentially locally algebraic representation $\phi_\ell$ is semisimple because by Frobenius reciprocity $\phi_\ell$ is a direct summand of the representation $\Ind_{\Gal_L}^{\Gal_K}(\phi_\ell|_{\Gal_L})$, which %and the latter 
%is semisimple by~(i) of Definition~\ref{def:plocalg} and because in characteristic zero induction preserves semisimplicity. Below in Proposition~\ref{pabconverse} we will show that any potentially locally algebraic $\phi_\ell$ is strongly $E$-rational for a suitable number field $E$. But it relies on \cite{DMOS82} and is not an elementary argument.}

\begin{prop}
Let $\phi_\ell:\Gal_K\to\GL_k(\overline\Q_\ell)$ be an  $\ell$-adic representation of a number field $K$.
\begin{enumerate}[(i)]
\item If $\phi_\ell$ is potentially locally algebraic, then $\phi_\ell$ is semisimple.
\item $\phi_\ell(\Gal_K)$ takes values in $\GL_n(F)$ for some finite extension $F\supset\Q_\ell$ in $\overline\Q_\ell$. In particular, 
there exists a number field $E$ and a place $\lambda$ of $E$ above $\ell$ such that $\phi_\ell(\Gal_K)\subset \GL_n(E_\lambda)$.
\end{enumerate}
\end{prop}
\begin{proof}
(i) Let $L$ be as in Definition~\ref{def:plocalg}. By Frobenius reciprocity $\phi_\ell$ is a direct summand of the representation $\Ind_{\Gal_L}^{\Gal_K}(\phi_\ell|_{\Gal_L})$. The latter is semisimple because in characteristic zero induction preserves semisimplicity, and $\phi_\ell|_{\Gal_L}$ is semisimple by~(i) of Definition~\ref{def:plocalg}.

(ii) This is general result for continuous representations of a compact group into $\GL_n(\overline\Q_\ell)$. It can be found in \cite[Corollary 5]{Dickinson}. The proof there is only for $\GL_2(\overline \Q_2)$. But the argument is general.
\end{proof}

%In Proposition~\ref{pabconverse} we will show that any potentially locally algebraic $\phi_\ell$ is strongly $E$-rational for a suitable number field $E$. But it %relies on \cite{DMOS82} and 
%is not an elementary argument.
%
%We remark that any $\lambda$-adic potentially locally algebraic representation $\phi_\lambda$ is semisimple because by Frobenius reciprocity $\phi_\lambda$ is a direct sum of the representation $\Ind_{\Gal_L}^{\Gal_K}(\rho_\lambda|_{\Gal_L})$ which is semisimple by~(i) of Definition~\ref{def:plocalg}. 

By Theorems \ref{rholambda} and \ref{SW}, the $\lambda$-adic representations \eqref{1stcs} attached to an
$E$-representation of $\bM_K^{\CM}$ are potentially locally algebraic. 
We will prove a converse result (Proposition~\ref{pabconverse}), which in particular, 
implies that any potentially locally algebraic $\phi_\ell$ is strongly $E$-rational for a suitable number field $E$. 
But it is not an elementary argument.
We first prove a lemma. %, we begin with the following lemma:

\begin{lem}\label{lem:RepsOverFinite}
Let $\bM$ be a pro-reductive group over a field $Q$ of characteristic zero 
and let $\rho:\bM\times F\to \GL_{d,F}$ be a representation over some extension field $F\supset Q$, with algebraic closure $\overline F$. Then there exists a finite extension $Q'\supset Q$ inside $\overline F$, and a representation $\rho':\bM\times Q'\to\GL_{d,Q'}$ such that $\rho\otimes_F\overline F=\rho'\otimes_{Q'}\overline{F}$.
\end{lem}

\begin{proof}
Without loss of generality, we assume $F=\overline F$.
Because $\rho$ is finite-dimensional, it factors via a reductive quotient of $\bM$ defined over $Q$, and so we assume that $\bM$ itself is reductive. Because $\rho$ is a direct summand of the semisimple representation $\Ind_{\bM^\circ}^{\bM} (\rho|_{\bM^\circ})$, we may assume that $\bM$ is connected reductive. By passing from $Q$ to a finite extension $Q'$, inside $F$, we may assume that $\bM$ is split reductive.
% and by passing from $F$ to a finite extension $F'$ inside $\overline F$, we may assume that 
As $\rho$ is a direct sum of (absolutely) irreducible representations, 
it suffices to prove the lemma for $\bM$ connected split reductive and $\rho$ (absolutely) irreducible. 
In this case $\rho$ is classified by highest weight theory \cite[Chapter 22]{Mi17}, and in particular $\rho$ is isomorphic to a representation defined over $Q'$. 
\end{proof}

We have a converse result.

\begin{prop}\label{pabconverse}
Let $\phi_{\lambda}:\Gal_K\to\GL_k(E_{\lambda})$ be a semisimple potentially locally algebraic $\lambda$-adic
representation of $K$. The following assertions hold.
\begin{enumerate}[(i)]
\item There exists an $E_{\lambda}$-representation 
$\phi_0: \bM_K^{\CM}\times E_{\lambda}\to \GL_{k,E_{\lambda}}$ such that 
$\phi_{\lambda}$ is equivalent to $\phi_0\circ s_{\ell}$, where $s_{\ell}$ is the map in Theorem \ref{Taniyama}(iii).
\item There exist a finite extension $E'/E$ and an $E'$-representation
$$\phi':\bM_K^{\CM}\times E'\to \GL_{k,E'}$$ 
such that $\phi_\lambda$ belongs to the semisimple potentially abelian strongly $E'$-rational compatible system
\begin{equation}\label{newcs}
\{\Gal_K  \stackrel{s_\ell}{\longrightarrow}\bM_K^{\CM}(\Q_\ell)
\stackrel{\phi'\otimes E'_{\lambda'}}{\rightarrow}\GL_k(E'_{\lambda'})\}_{\lambda'\in\Sigma_{E'}}.
\end{equation}
\end{enumerate}
\end{prop}

\begin{proof}
(i). When $K=\Q$, this is established in \cite[IV, Proposition D.1]{DMOS82}.
The general case follows from the $\Q$-case: Let $\phi_0':\bM_\Q^\CM\times E_\lambda \to\GL_{k',E_\lambda}$ be an $E_\lambda$-representations
such that $\phi_0'\circ s_\ell\simeq\text{Ind}_K^\Q(\phi_\lambda)$ for $k'=k\cdot[K:\Q]$. 
Then $\phi_0'\circ s_\ell|_{\Gal_K} \simeq\text{Res}^\Q_K \text{Ind}_K^\Q(\phi_\lambda)$, and moreover $s_\ell|_{\Gal_K}:\Gal_K\to\bM_K^\CM(\Q_\ell)$ 
is the homomorphism from Theorem \ref{Taniyama}(iii).

Now $\phi_\lambda$ is a direct summand of $\Phi:=\text{Res}^\Q_K \text{Ind}_K^\Q(\phi_\lambda)$. So with respect to a suitable basis of the underlying vector space
$E_\lambda^{k'}$ of $\Phi$, the image of $\Gal_K$ preserves the subspace $E_\lambda^{k'-k}\oplus 0^k$, and the action on the quotient is that given by $\phi_\lambda$.
By the Zariski density of $s_\ell(\Gal_K)$ in $\bM^\CM_K$, it follows that the action of $\bM^\CM_K\times E_\lambda$ preserves the same subspace, and hence that there is a 
quotient homomorphism $\phi_0:\bM^\CM_K\times E_\lambda\to\GL_{k,E_\lambda}$, such that $\phi_0\circ s_\ell|_{\Gal_K}=\phi_\lambda$, proving (i).

%The general case follows from the $\Q$-case, 
%(applying to the potentially locally algebraic $\text{Ind}_K^\Q(\phi_\lambda)$), 
%the fact that $\phi_\lambda$ is a subrepresentation 
%of $\text{Res}^\Q_K (\text{Ind}_K^\Q(\phi_\lambda))$, and the construction/properties of 
%$s_\ell:\Gal_K\to \bM_K^{\CM}(\Q_\ell)$ in Theorem \ref{Taniyama}(i),(ii),(iii). 

%(ii). {\color{blue} Let $\phi_0$ be as in (i). Any finite-dimensional representation of a pro-reductive group has a model over a finite extension of the field of definition of the group.\footnote{{\color{blue} Should one give more detail: Reduction to $\rho$ absolutely irreducible, to $G$ connected reductive, to highest weight theory? Or is there a reference, say for $G$ reductive?}} Hence there exist a finite extension $E'/E$, an $E'$-representation
%$$\phi':\bM_K^{\CM}\times E'\to \GL_{k,E'},$$ 
%and a place $\lambda'\in\Sigma_{E'}$ (dividing $\lambda$) 
%such that $\phi_0\times E'_{\lambda'}=\phi'\times E'_{\lambda'}$.
%We are done by (i) and Theorem \ref{rholambda}.}
%

Given $\phi_0$ in (i), by Lemma~\ref{lem:RepsOverFinite} there exist a finite extension $E'/E$, an $E'$-representation
$$\phi':\bM_K^{\CM}\times E'\to \GL_{k,E'},$$ 
and a place $\lambda'\in\Sigma_{E'}$ (dividing $\lambda$) 
such that $\phi_0\otimes E'_{\lambda'}=\phi'\otimes E'_{\lambda'}$.
We are done by (i) and Theorem \ref{rholambda}.
\end{proof}

Below is the main result in this section.

\begin{thm}\label{pabindep} 
Let $\{\phi_\lambda:\Gal_K\to\GL_k(E_\lambda)\}_{\Sigma_E}$ be a semisimple potentially abelian
(strongly) $E$-rational compatible system with $\{\bG_{\phi_\lambda}\}_{\Sigma_E}$
as the system of algebraic monodromy groups. 
Then for some $d>0$ there exists an $E$-morphism $\Phi:\bM_K^{\CM}\times E\to \GL_{kd,E}$ such that 
  $\{\phi_{\lambda}^{\oplus d}\}_{\Sigma_{E}}$ is the $E$-compatible system attached to $\Phi$ as \eqref{1stcs}.
In particular,   %$\{\phi_{\lambda}^{\oplus d}\}_{\Sigma_{E}}$ is strongly $E$-compatible and 
\begin{equation}\label{pabgpindep}
\Phi(\bM_K^{\CM}\times E)\times E_\lambda \simeq \bG_{\phi_\lambda}
\hspace{.3in}\text{for all} ~\lambda\in\Sigma_E.
\end{equation}
%There exists an $E$-morphism $\phi:\bM_K^{\CM}\times E\to \GL_{k,E}$ such that 
%  $\{\phi_{\lambda}\}_{\Sigma_{E}}$ is the $E$-compatible system \eqref{1stcs} attached to $\phi$. In particular,
%\begin{equation}\label{pabgpindep}
%[\phi(\bM_K^{\CM}\times E)\subset\GL_{k,E}] \simeq [\bG_{\phi_\lambda}\subset \GL_{k,E_\lambda}] \hspace{.3in}\text{for all} ~\lambda\in\Sigma_E.
%\end{equation}
\end{thm}

\begin{proof}
By Theorem \ref{SW}, every $\phi_\lambda$ is potentially locally algebraic.
Thus by Proposition \ref{pabconverse}(ii) and Theorem \ref{unique}, there exist a finite extension $E'/E$
and an $E'$-morphism $\phi':\bM_K^{\CM}\times E'\to \GL_{k,E'}$ such that 
\begin{equation}\label{twocs}
\{\phi_\lambda\otimes E'_{\lambda'}\}_{\Sigma_{E'}}\simeq\{(\phi'\otimes E'_{\lambda'})\circ s_\ell\}_{\Sigma_{E'}}.
\end{equation}
%by Proposition \ref{pabconverse}(ii) and Theorem \ref{unique}. 
 By applying restriction of scalars to \eqref{twocs}
with respect to $E'/E$ and using Proposition \ref{see}, we obtain
\begin{equation}\label{long=}
\phi_\lambda^{\oplus [E':E]}\simeq \prod_{\lambda'|\lambda}(\phi_\lambda\otimes E'_{\lambda'})
\simeq \prod_{\lambda'|\lambda} (\phi'\otimes E'_{\lambda'})\circ s_\ell=(\mathrm{Res}_{E'/E}(\phi')\otimes E_\lambda)\circ s_\ell
\end{equation}
for every $\lambda$.
By taking $d=[E':E]$, $\Phi=\mathrm{Res}_{E'/E}(\phi')$,  the first assertion follows from \eqref{long=}.
Since $s_\ell(\Gal_K)$ is Zariski dense in $\bM_K^{\CM}$ (Theorem \ref{Taniyama}(iii))
and the algebraic monodromy groups of $\phi_\lambda^{\oplus d}$ and $\phi_\lambda$ are isomorphic, we obtain \eqref{pabgpindep}.
\end{proof}

\begin{remark}
 Suppose $\rho:\Gal_K\to\GL_{k}(\overline\Q)$ is an Artin representation with trace field $E$. Then it need not be the case that $\rho$ can be defined over $E$, as can be deduced for instance from \cite[Section 12]{Se77}. In particular the associated compatible system is $E$-rational, but not strongly $E$-rational. Therefore one cannot expect the conclusion of Theorem~\ref{pabindep} to hold with $d=1$, unlike in the abelian case; cf.~Corollary~\ref{cor:abelian-strongly}.
% By arguments similar to the proof of Theorem~\ref{pabindep}, one can also show the following: Let $(\phi_\lambda)_{\Pi_E}$ be any $E$-rational compatible system that becomes strongly $E'$-rational over a finite extension $E'$ of $E$. Then $(\phi_\lambda^{\oplus [E':E]})_{\Pi_E}$ is strongly $E$-rational.{\color{magenta} Expand?}
\end{remark}

\section{Proof of Theorem \ref{main1}}\label{s6}
Building on Definition~\ref{def:Glambda-Erat}, we associate, the following $E_\lambda$-reductive groups to a semisimple $E$-rational compatible system $\{\rho_\lambda:\Gal_K\to\GL_n(\Elambda)\}_{\Pi_E}$.
\begin{itemize}
\item $\bG_\lambda$: the algebraic monodromy group of $\rho_\lambda$ defined over $E_\lambda$ with morphism 
$\Gal_K\stackrel{\rho_\lambda}{\rightarrow}\bG_\lambda(E_\lambda)$.
\item $\bG_\lambda^{\cm}:=\bG_\lambda/[\bG_\lambda^\circ,\bG_\lambda^\circ]$, the maximal %\textcolor[rgb]{0,0,1}
{potentially abelian} quotient of $\bG_\lambda$.
\item $\bG_\lambda^{\ab}:=\bG_\lambda/[\bG_\lambda,\bG_\lambda]$, the maximal abelian quotient of $\bG_\lambda$.
\end{itemize}
By Proposition \ref{prop:Glambda-andBC}, we can write
\begin{itemize}
\item $\bG_{\lambda,\overline E_\lambda}$ as the algebraic monodromy group of $\rho_\lambda$.
\end{itemize}

%To prove Theorem \ref{main1}, one may replace $\{\rho_\lambda\}_{\Pi_E}$ with its coefficient extension with respect to some finite extension $E'/E$.

\subsection{Auxiliary compatible systems}\label{s6.1}
\begin{prop}\label{aux}
Let  $\{\rho_\lambda\}_{\Pi_E}$ be a semisimple
$E$-rational compatible system and let $\lambda_0\in\Pi_E$.
\begin{enumerate}[(i)]
\item There exists a semisimple potentially abelian  strongly $E$-rational compatible system $\{\phi_{\lambda}^{\cm}\}_{\Sigma_E}$
of $K$ and
a family of $E_{\lambda}$-representations $\{f_\lambda:\bG_{\lambda}^{\cm}\hto\GL_{k,E_{\lambda}}\}_{\Pi_E}$
such that 
$$\phi_{\lambda}^{\cm}\simeq [\Gal_K\stackrel{\rho_{\lambda}}{\rightarrow}\bG_\lambda(E_{\lambda})
\to \bG_{\lambda}^{\cm}(E_{\lambda})\stackrel{f_\lambda}{\to}\GL_k(E_{\lambda})] \hspace{.3in}\text{for all}~ \lambda\in\Pi_E$$
 and $f_{\lambda_0}$ is faithful.
\item There exist a semisimple abelian  strongly $E$-rational compatible system $\{\phi_{\lambda}^{\ab}\}_{\Sigma_E}$
of $K$ and
a family of $E_{\lambda}$-representations $\{r_\lambda:\bG_{\lambda}^{\ab}\hto\GL_{m,E_{\lambda}}\}_{\Pi_E}$
such that 
$$\phi_{\lambda}^{\ab}\simeq [\Gal_K\stackrel{\rho_{\lambda}}{\rightarrow}\bG_\lambda(E_{\lambda})
\to \bG_{\lambda}^{\ab}(E_{\lambda})\stackrel{r_\lambda}{\to}\GL_m(E_{\lambda})] \hspace{.3in}\text{for all}~ \lambda\in\Pi_E$$
 and $r_{\lambda_0}$ is faithful.
\end{enumerate}
\end{prop}

\begin{proof}
We demonstrate assertion (i) only since (ii) is completely analogous.
%Let $F$ be a finite extension of $E_{\lambda_0}$ over which $\bG_{\rho_\lambda,F}$ is a subgroup scheme of $\GL_{n,F}$, and 
Denote by 
$V_{\lambda_0}$ the $\overline E_{\lambda_0}$-representation space of $\rho_{\lambda_0}$.
Pick a faithful $E_{\lambda_0}$-representation of $\bG_{\lambda_0}^{\cm}$:
$$f_{\lambda_0}': \bG_{\lambda_0}^{\cm}\to \GL_{d, E_{\lambda_0}}.$$
Since $\bG_{\lambda_0}$ is reductive, it follows from Proposition \ref{prop:Glambda-andBC} and \cite[I. Proposition 3.1(a)]{DMOS82}
that the natural representation
$$\bG_{\lambda_0}\to\bG_{\lambda_0}^{\cm}\stackrel{f_{\lambda_0}'}{\rightarrow} \GL_{d, E_{\lambda_0}},$$
after base change to $\overline E_{\lambda_0}$, is a $\bG_{\lambda_0,\overline E_{\lambda_0}}$-subrepresentation of 
$$\bigoplus_{1\leq j\leq t} (V_{\lambda_0}^{\otimes a_j}\otimes V_{\lambda_0}^{\vee, \otimes b_j})$$
for some $t\in\N$ and $a_j,b_j\in\Z_{\geq 0}$ for $1\leq j\leq t$.
Define the semisimple $\lambda_0$-adic representation
\begin{equation}\label{constructphi}
\phi_{\lambda_0}':\Gal_K \stackrel{\rho_{\lambda_0}}{\rightarrow}\bG_{\lambda_0}(E_{\lambda_0})
\stackrel{}{\rightarrow} \bG_{\lambda_0}^{\cm}(E_{\lambda_0})  \stackrel{f_{\lambda_0}'}{\rightarrow}\GL_d(E_{\lambda_0}).
\end{equation}
It follows that 
\begin{itemize}
\item $\phi_{\lambda_0}'$ is potentially abelian, and
\item $\phi_{\lambda_0}'\otimes\overline E_{\lambda_0}$ is a subrepresentation of the $E$-rational $\bigoplus_{1\leq j\leq t} (\rho_{\lambda_0}^{\otimes a_j}\otimes \rho_{\lambda_0}^{\vee, \otimes b_j})$. %, when considered as $F$-representations.
\end{itemize}

These conditions and Theorem \ref{BH1} imply that $\phi_{\lambda_0}'$ is potentially locally algebraic (subsection \ref{s5.3}).
By Proposition \ref{pabconverse}(ii),
there exist a finite extension $E'/E$ and an $E'$-representation
$$\phi':\bM_K^{\CM}\times E'\to \GL_{d,E'}$$ 
giving rise to a semisimple potentially abelian strongly $E'$-rational compatible system
\begin{equation}\label{newcs2}
\{ \phi'_{\lambda'}:\Gal_K   \to \GL_{d}(E'_{\lambda'})\}_{\lambda'\in\Sigma_{E'}}.
\end{equation}
containing the $\lambda_0$-adic representation $\phi_{\lambda_0}'$, i.e., one has 
$\phi'_{\lambda_0'}=\phi_{\lambda_0}'\otimes E'_{\lambda_0'}$
for some $\lambda_0'\in\Sigma_{E'}$ dividing $\lambda_0$. Then Weil restriction 
as discussed in subsection \ref{sExt} converts $\phi'$ into a morphism 
$$\phi^\cm:\bM_K^{\CM}\times E\to \mathrm{Res}_{E'/E}\GL_{d,E'}\subset \GL_{d[E':E]},$$ 
that induces an strongly $E$-rational potentially abelian compatible system 
\begin{equation}\label{newcs3}
\{ \phi_{\lambda}^\cm:\Gal_K   \to \GL_{d[E':E]}(E_{\lambda})\}_{\lambda\in\Sigma_{E}}
\end{equation}
and as in the proof of Theorem~\ref{pabindep}, one has $\phi_{\lambda_0}^\cm\simeq(\phi'_{\lambda_0})^{[E':E]}$. 
Let $\{\bG_{\phi_\lambda^{\cm}}\}$ 
be the system of algebraic monodromy groups of $\{\phi_\lambda^{\cm}\}$.
Observe that by the faithfulness of $f'_{\lambda_0}$, one has $\bG_{\lambda_0}^\cm\simeq \bG_{\phi_{\lambda_0}^\cm}$.

Consider the semisimple $E$-rational compatible system 
given by direct sum:
\begin{equation}\label{sumtwosys}
\{\rho_\lambda\oplus\phi_{\lambda}^\cm:\Gal_K\stackrel{}{\rightarrow}  
\bG_\lambda(E_\lambda)\times \bG_{\phi_\lambda^{\cm}}(E_\lambda)\hto \GL_n(\overline E_\lambda)\times\GL_{d[E':E]}(\overline E_\lambda)\}_{\Pi_E},
\end{equation}
where the injection $\bG_\lambda(E_\lambda)\hto \GL_n(\overline E_\lambda)$
is given by $j_{E_\lambda}$ in Proposition \ref{prop:Glambda-andBC}(i).
Let $\{\bH_{\lambda}\}_{\Pi_E}$ be the system of reductive groups defined over $E_\lambda$
given by the Zariski closure of the image of $\Gal_K$ in $\bG_\lambda\times \bG_{\phi_\lambda^{\cm}}$.
It follows that $\{\bH_{\lambda}\times \overline E_\lambda\}_{\Pi_E}$
is the system of algebraic monodromy groups of \eqref{sumtwosys}.
Let $\pi_1:\bG_\lambda\times \bG_{\phi_\lambda^{\cm}} \to\bG_\lambda$ 
and $\pi_2:\bG_\lambda\times \bG_{\phi_\lambda^{\cm}}\to\bG_{\phi_\lambda^{\cm}}$ be the natural projections.
For each $\lambda\in \Pi_E$,  we have a natural surjection 
\begin{equation}\label{natonto}
\pi_1: \bH_{\lambda}\to \bG_\lambda,
\end{equation}
which is an isomorphism when $\lambda=\lambda_0$ by construction \eqref{constructphi}.
Hence, \eqref{natonto} is an isomorphism for all $\lambda\in \Pi_E$, 
 because by \cite[Corollary 3.12]{HL25}
it is so after base change to $\Elambda$. Therefore we can define
\begin{equation}\label{auxmap}
\bG_\lambda\stackrel{\pi_1^{-1}}{\rightarrow}\bH_{\lambda}\stackrel{\pi_2}{\rightarrow}\bG_{\phi_\lambda^{\cm}}\subset \GL_{d[E':E], E_{\lambda}}
\end{equation}
for all $\lambda$. 

Since $\pi_2(\bH_{\lambda})=\bG_{\phi_\lambda^{\cm}}$ is potentially abelian,
\eqref{auxmap} factors through (the maximal %\textcolor[rgb]{0,0,1}
{potentially abelian} quotient) $\bG_\lambda^{\cm}$ for all $\lambda$
and we obtain an $E_{\lambda}$-representation
\begin{equation}\label{auxmapf}
f_{\lambda}: \bG_\lambda^{\cm} \to \GL_{d[E':E], E_{\lambda}}
\end{equation}
such that 
\begin{equation}\label{factor}
\phi_{\lambda}^\cm\simeq [\Gal_K\stackrel{\rho_{\lambda}}{\rightarrow}\bG_\lambda(E_{\lambda})
\to \bG_{\lambda}^{\cm}(E_{\lambda})\stackrel{f_{\lambda}}{\to}\GL_{d[E':E]}(E_{\lambda})]
\end{equation}
for all $\lambda\in \Pi_E$, and such that $f_{\lambda_0}=(f'_{\lambda_0})^{\oplus[E':E]}$, and in particular $f_{\lambda_0}$ is faithful. 
\end{proof}

\subsection{Completion of the proof}\label{s6.2}
Let $\{\bG_{\phi_\lambda^{\cm}}\}$ and $\{\bG_{\phi_\lambda^{\ab}}\}$
be the systems of algebraic monodromy groups of $\{\phi_\lambda^{\cm}\}$ and $\{\phi_\lambda^{\ab}\}$
in Proposition \ref{aux}.
If we prove Proposition \ref{auxgp} below,
then Theorem \ref{main1}(i),(ii) hold and we obtain
$$\bG_\lambda^{\cm}\simeq \bG_{\phi_\lambda^{\cm}} \hspace{.3in}\text{and}\hspace{.3in} \bG_\lambda^{\ab}\simeq \bG_{\phi_\lambda^{\ab}}\hspace{.3in}\text{for all}~ \lambda\in\Pi_E.$$
This, together with Theorems \ref{pabindep} and \ref{converse}, imply Theorem \ref{main1}(iii).

\begin{prop}\label{auxgp}
The $E_\lambda$-representations $f_\lambda$ and $r_\lambda$ in Proposition \ref{aux} are faithful for all $\lambda\in\Pi_E$.
\end{prop}

\begin{proof}
Again, we consider $\{f_\lambda\}$ only since the other case is completely analogous.
By Proposition \ref{aux}, 
we obtain the surjection (still denoted by $f_\lambda$ for simplicity)
\begin{equation}\label{surj0}
f_{\lambda}:\bG_{\lambda}^{\cm}\twoheadrightarrow\bG_{\phi_{\lambda}^{\cm}}
\end{equation}
for all $\lambda$.
Let $\lambda_1\in\Pi_E$ be distinct from $\lambda_0$. It suffices to show that $f_{\lambda_1}$ is an isomorphism.

By applying Proposition \ref{aux} to the pair $(\{\rho_\lambda\},\lambda_1)$, we obtain similarly 
a ($h$-dimensional) semisimple %\textcolor[rgb]{0,0,1}
{potentially abelian}  $E$-rational compatible system $\{\psi_{\lambda}^{\cm}\}_{\Sigma_E}$
with  $\{\bG_{\psi_{\lambda}^{\cm}}\}_{\Sigma_E}$ as the system of algebraic monodromy groups, and 
also surjections
\begin{equation}\label{surj1}
g_{\lambda}:\bG_{\lambda}^{\cm}\twoheadrightarrow\bG_{\psi_{\lambda}^{\cm}}
\end{equation}
for $\lambda\in\Pi_E$.
We record the properties for the four families $\{\phi_{\lambda}^{\cm}\}$, $\{f_\lambda\}$, 
$\{\psi_{\lambda}^{\cm}\}$, and $\{g_\lambda\}$.

\begin{prop}\label{record}
\begin{enumerate}[(i)] 
\item For $\lambda\in\Pi_E$, the potentially abelian $\phi_{\lambda}^{\cm}$ and $\psi_{\lambda}^{\cm}$ are equivalent to, respectively,
$$\Gal_K\stackrel{\rho_\lambda}{\rightarrow}\bG_\lambda(E_\lambda)\to \bG_\lambda^{\cm}(E_\lambda)
\stackrel{f_{\lambda}}{\rightarrow}\bG_{\phi_{\lambda}^{\cm}}(E_\lambda)\subset\GL_k(E_\lambda)$$
and 
$$\Gal_K\stackrel{\rho_\lambda}{\rightarrow}\bG_\lambda(E_\lambda)\to \bG_\lambda^{\cm}(E_\lambda)
\stackrel{g_{\lambda}}{\rightarrow}\bG_{\psi_{\lambda}^{\cm}}(E_\lambda)\subset\GL_h(E_\lambda).$$
\item The surjection $f_{\lambda}$ (resp. $g_{\lambda}$) is an isomorphism when $\lambda=\lambda_0$ (resp. $\lambda=\lambda_1$).
\item The kernel of $f_{\lambda}$ (resp. $g_{\lambda}$) belongs to  the identity component $\bG_{\lambda}^{\cm,\circ}$ of $\bG_{\lambda}^{\cm}$.
\end{enumerate}
\end{prop}

\begin{proof}
By Proposition \ref{aux}, (i) and (ii) are obvious.
To prove (iii), consider only $f_{\lambda}$.
It suffices to 
show that the component groups of $\bG_{\lambda}^{\cm}$
and $\bG_{\phi_{\lambda}^{\cm}}$ are isomorphic for all $\lambda$.
Since the component groups
$$\pi_0(\bG_{\lambda}^{\cm})\simeq \pi_0(\bG_\lambda)\simeq \pi_0(\bG_{\lambda,\overline E_\lambda})\hspace{.3in}\text{and}\hspace{.31in}
\pi_0(\bG_{\phi_{\lambda}^{\cm}})$$
are independent of $\lambda$ by \cite{Se81}, we are done by (ii).
\end{proof}

Take a finite extension $L/K$ such that $\{\phi_{\lambda}^{\cm}|_{\Gal_L}\}$ 
and $\{\psi_{\lambda}^{\cm}|_{\Gal_L}\}$ are abelian semisimple (strongly) $E$-rational compatible systems,
with $\{\bG_{\phi_{\lambda}^{\cm}}^\circ\}$ and $\{\bG_{\psi_{\lambda}^{\cm}}^\circ\}$ respectively 
as the systems of algebraic monodromy groups \cite{Se81}.
By Proposition \ref{record}(i),(ii), there exist $E_{\lambda_0}$-morphism
$$\alpha_{\lambda_0}:\bG_{\phi_{\lambda_0}^{\cm}}^\circ \stackrel{f_{\lambda_0}^{-1}}{\rightarrow} 
\bG_{\lambda_0}^{\cm,\circ}\stackrel{g_{\lambda_0}}{\rightarrow}\bG_{\psi_{\lambda_0}^{\cm}}^\circ$$
and $E_{\lambda_1}$-morphism
$$\beta_{\lambda_1}:\bG_{\psi_{\lambda_1}^{\cm}}^\circ \stackrel{g_{\lambda_1}^{-1}}{\rightarrow} 
\bG_{\lambda_1}^{\cm,\circ}\stackrel{f_{\lambda_1}}{\rightarrow} \bG_{\phi_{\lambda_1}^{\cm}}^\circ,$$
that are both surjective, such that 
$$\psi_{\lambda_0}^{\cm}|_{\Gal_L}\simeq [\Gal_L\stackrel{\phi_{\lambda_0}^{\cm}}{\rightarrow}\bG_{\phi_{\lambda_0}^{\cm}}^\circ(E_{\lambda_0})
\stackrel{\alpha_{\lambda_0}}{\rightarrow} \bG_{\psi_{\lambda_0}^{\cm}}^\circ(E_{\lambda_0})\subset\GL_h(E_{\lambda_0})]$$
and
$$\phi_{\lambda_1}^{\cm}|_{\Gal_L}\simeq [\Gal_L\stackrel{\psi_{\lambda_1}^{\cm}}{\rightarrow}\bG_{\psi_{\lambda_1}^{\cm}}^\circ(E_{\lambda_1})
\stackrel{\beta_{\lambda_1}}{\rightarrow} \bG_{\phi_{\lambda_1}^{\cm}}^\circ(E_{\lambda_1})\subset\GL_k(E_{\lambda_1})].$$
Then Proposition \ref{extendsys} implies the existence of two systems 
$\{\alpha_{\lambda}:\bG_{\phi_{\lambda}^{\cm}}^\circ \to \bG_{\psi_{\lambda}^{\cm}}^\circ\}$
and $\{\beta_{\lambda}:\bG_{\psi_{\lambda}^{\cm}}^\circ \to \bG_{\phi_{\lambda}^{\cm}}^\circ\}$
 of surjective $E_\lambda$-morphisms such that 
$$[\Gal_L\stackrel{\psi_{\lambda}^{\cm}}{\rightarrow}\bG_{\psi_{\lambda}^{\cm}}^\circ(E_\lambda)\subset\GL_h(E_\lambda)]
\simeq[\Gal_L\stackrel{\phi_{\lambda}^{\cm}}{\rightarrow}\bG_{\phi_{\lambda}^{\cm}}^\circ(E_\lambda)
\stackrel{\alpha_\lambda}{\rightarrow}\bG_{\psi_{\lambda}^{\cm}}^\circ(E_\lambda)\subset\GL_h(E_\lambda)]$$
and 
$$[\Gal_L\stackrel{\phi_{\lambda}^{\cm}}{\rightarrow}\bG_{\phi_{\lambda}^{\cm}}^\circ(E_\lambda)\subset\GL_k(E_\lambda)]
\simeq [\Gal_L\stackrel{\psi_{\lambda}^{\cm}}{\rightarrow}\bG_{\psi_{\lambda}^{\cm}}^\circ(E_\lambda)
\stackrel{\beta_\lambda}{\rightarrow}\bG_{\phi_{\lambda}^{\cm}}^\circ(E_\lambda)\subset\GL_k(E_\lambda)]$$
for all $\lambda$. We claim that we can modify $\alpha_\lambda$ and $\beta_\lambda$, so that in fact we have equality: 
Let us explain this for $\alpha_\lambda$, the other case being analogous. Let $\gamma_\lambda\in\GL_h(E_\lambda)$ be such that,
denoting by $c_{\gamma_\lambda}$ conjugation by $\gamma_\lambda$, we have  
$\psi_{\lambda}^{\cm}= c_{\gamma_\lambda}\circ\alpha_\lambda\circ \phi_\lambda^\cm$. Because $\psi_{\lambda}^{\cm}(\Gal_L)$ and 
$\phi_\lambda^\cm(\Gal_L)$ have Zariski dense image and $\alpha_\lambda$ is surjective, the conjugation map $c_{\gamma_\lambda}$
induces an automorphism of $\bG_{\psi_{\lambda}^{\cm}}$, and we can replace $\alpha_\lambda$ by $c_{\gamma_\lambda}\circ\alpha_\lambda$,
proving the claim.

It follows that we can assume $\psi_{\lambda}^{\cm}= \alpha_\lambda\circ \phi_\lambda^\cm$ and 
$\phi_{\lambda}^{\cm}= \beta_\lambda\circ \psi_\lambda^\cm$ for all $\lambda$.
Again because $\psi_{\lambda}^{\cm}(\Gal_L)$ has Zariski dense image, 
we deduce $\alpha_{\lambda_1}\circ\beta_{\lambda_1}=\mathrm{Id}$, and hence that $\beta_{\lambda_1}$ is an isomorphism,
and hence by Proposition \ref{record}(iii), that $f_{\lambda_1}$ is an isomorphism. 
We are done.%\footnote{{\color{magenta} G: please check if this rewrite is okay. I felt safer having equalities and not only isomorphisms.}}
\end{proof}

\section{Proof of Theorem \ref{main2}}\label{s6}
%\footnote{{\color{magenta} G: I have not tried to generalize PSW from strongly $E$-rational to $E$-rational semisimple compatible systems.}}
We will follow Patrikis-Snowden-Wiles' proof of Theorem \ref{main2} in \cite[$\mathsection3.5$]{PSW18}.
Firstly, we define the necessary notation.
Secondly, we present the main ingredients of their proof.
Finally, we describe how the strongly $E$-compatible condition and the Hodge-Tate condition (HT) in subsection \ref{s1.3} can be omitted in the proof
by applying Theorems \ref{coeffd}(i) and \ref{main1}.

\subsection{Some notation}\label{s6.1}
Let $F/\Q_\ell$ be a finite field extension, $\bG$ be a reductive group defined over $F$,
and $\Gamma\subset \bG(F)$ be a profinite subgroup.
Define the following groups.
\begin{itemize}
\item $\bG^\circ$: the identity component of $\bG$.
\item $\bG^{\ad}$: the adjoint quotient of $\bG^{\circ}$ with $\tau: \bG^{\circ}(F)\to \bG^{\ad}(F)$ the canonical map.
\item $\bG^{\sc}$: the universal cover of $\bG^{\ad}$ with $\sigma: \bG^{\sc}(F)\to \bG^{\ad}(F)$ the canonical map.
\item $\bG^{\der}$: the derived group $[\bG^\circ,\bG^\circ]$.
\item $\bG^{\mathrm{tor}}$: the torus $\bG^\circ/\bG^{\der}$.
\item $\Gamma^\circ$: the intersection $\Gamma\cap\bG^\circ(F)$.
\item $\Gamma^{\mathrm{tor}}$: the image of $\Gamma^\circ$ in $\bG^{\mathrm{tor}}(F)$.
\end{itemize}

\subsection{Main ingredients}\label{s6.2}
The first ingredient  is purely group/representation theoretic.

\begin{thm}\cite[Theorem 3.8]{PSW18} \label{gprepn}
Let $\bG/F$ be a reductive group with $\bG^\circ$ unramified, and let $\Gamma\subset\bG(F)$ be 
a Zariski dense  profinite group such that $\sigma^{-1}(\tau(\Gamma^\circ))$ is
hyperspecial in $\bG^{\sc}(F)$. Assume that $\Gamma^{\mathrm{tor}}$
is open in $\bG^{\mathrm{tor}}(F)$ and its index in the maximal compact is small
compared to $\ell$. Also assume that $\dim(\bG)$ and $\# \pi_0(\bG)$ are small 
compared to $\ell$.

Let $(\rho,V)$ be a representation of $\bG$ over $\overline F$ with $\dim(V)$ small compared to $\ell$.
Then there exists a $\Gamma$-stable lattice $\mathcal V$ in $V$ such that the natural map
$$\End_\Gamma(\mathcal V)\otimes \overline\F_\ell \to \End_\Gamma(\mathcal V\otimes \overline\F_\ell)$$
is an isomorphism.
\end{thm}

The second one is a big Galois image result of Larsen on the ``semisimple part''.

\begin{thm}\cite{La95}\label{Larsen}
Let $\{\rho_\ell:\Gal_K\to\GL_k(\Q_\ell)\}_{\Pi}$ be a semisimple (strongly) $\Q$-rational compatible system
of $K$ indexed by a subset $\Pi\subset\Sigma_\Q$ of Dirichlet density one,
with $\{\bG_\ell\}_{\Pi}$ as the system of algebraic monodromy groups.
Then there exists a Dirichlet density one subset $\Pi'\subset\Pi$ such that 
$\sigma^{-1}(\tau(\mathrm{Im}(\rho_\ell)^\circ))$ is
hyperspecial in $\bG_\ell^{\sc}(\Q_\ell)$ for all $\ell\in\Pi'$.
\end{thm}

The last one is a big Galois image result on the ``toric part''.

\begin{thm}\cite[Theorem 1.6]{PSW18}\label{toric}
Let $\{\rho_\ell\Gal_K\to\GL_k(\Q_\ell)\}_{\Pi}$ be a semisimple (strongly) $\Q$-rational compatible system
of $K$ indexed by $\Pi\subset\Sigma_\Q$ that satisfies 
the Hodge-Tate condition (HT) in subsection \ref{s1.3},
with $\{\bG_\ell\}_{\Pi}$ as the system of algebraic monodromy groups.
Then there exists a positive integer $N$ with the following property:
the index of $\mathrm{Im}(\rho_\ell)^{\mathrm{tor}}$ in the maximal compact subgroup of $\bG_\ell^{\mathrm{tor}}(\Q_\ell)$
is at most $N$ for all $\ell\in\Pi$.
\end{thm}

\subsection{Proof of Theorem~\ref{main2}; assembling the pieces}\label{s6.3}
Let $\Pi\subset\Sigma_\Q$ be a subset of Dirichlet density one and  
$$\{\rho_\lambda: \Gal_K\to \GL_n(\overline E_\lambda)\}_{\Sigma_E(\Pi)}$$ 
be a  semisimple $E$-rational compatible system of $K$
indexed by $\Sigma_E(\Pi)$. 
By Theorem \ref{coeffd}(i), we may assume $\rho_\lambda(\Gal_K)\subset\GL_n(E_\lambda)$ for all $\lambda$.
By restriction of scalars with respect to $E/\Q$ (subsection \ref{sExtcs}),
we obtain a semisimple strongly $\Q$-rational compatible system of $K$
\begin{equation}\label{Qcs}
\{\rho_\ell:=\prod_{\lambda|\ell}\rho_\lambda:\Gal_K\to\GL_{n[E:\Q]}(\Q_\ell)\}_{\Pi}
\end{equation}
 indexed by $\Pi$.
According to \cite[$\mathsection3.5$]{PSW18}, if the assertions of Theorems \ref{Larsen} and \ref{toric}
hold for \eqref{Qcs}, then  Theorem \ref{gprepn} and the $\ell$-independence of $\pi_0(\bG_\ell)$ \cite{Se81} imply Theorem \ref{main2}.

Since Theorem \ref{Larsen} works for any semisimple strongly $\Q$-rational compatible system,
it remains to establish the assertion of Theorem \ref{toric} for \eqref{Qcs}.
By Serre \cite{Se81}, we may assume $\bG_\ell$ is connected for all $\ell$.
Attached to $\{\rho_\ell\}_{\Pi}$ is an abelian strongly
$\Q$-compatible system $\{\phi_\ell^{\ab}\}_\Pi$ in Theorem \ref{main1}(ii).
By Theorem \ref{main1}(ii), we obtain $\bG_\ell^{\mathrm{tor}}=\bG_\ell^{\ab}$ and
the equivalence of the assertions of Theorem \ref{toric} on $\{\rho_\ell\}_{\Pi}$ and $\{\phi_\ell^{\ab}\}_\Pi$. 
Since $\{\phi_\ell^{\ab}\}_\Pi$ satisfies (HT) automatically, 
we are done by applying Theorem \ref{toric} to $\{\phi_\ell^{\ab}\}_\Pi$.\qed

\vspace{.2in}

\section*{Acknowledgments}
The authors thank Ga\"etan Chenevier, Toby Gee, and Alireza Shavali for their comments and suggestions on this work. Additionally, we thank Michael Larsen, Stefan Patrikis, and Andrew Wiles for their interest.
 G.B. received funding by the Deutsche Forschungsgemeinschaft (DFG, German Research Foundation) TRR 326 \textit{Geometry and Arithmetic of Uniformized Structures}, project number 444845124. C.-Y.H. was partially supported by Hong Kong RGC and a Humboldt Research Fellowship.

\vspace{.1in}
\end{document}